\documentclass[reqno,a4paper,12pt]{amsart}
\usepackage{latexsym, amsmath, amssymb, amsthm, a4}

\usepackage{amsfonts, bbold, bbm, marvosym}

\usepackage{mathrsfs}

\usepackage[mathscr]{eucal}
\font\sc=rsfs10 at 12pt

\renewcommand{\a}{\alpha}
\newcommand{\be}{\beta}

\newcommand{\G}{\Gamma}
\newcommand{\de}{\delta}
\newcommand{\D}{\Delta}
\newcommand{\e}{\epsilon}
\newcommand{\ve}{\varepsilon}
\newcommand{\z}{\zeta}
\newcommand{\y}{\eta}

\newcommand{\io}{\iota}

\newcommand{\la}{\lambda}

\newcommand{\m}{\mu}
\newcommand{\n}{\nu}
\newcommand{\x}{\xi}
\newcommand{\X}{\Xi}
\newcommand{\ro}{\rho}
\newcommand{\s}{\sigma}
\newcommand{\Si}{\Sigma}
\newcommand{\vs}{\varsigma}
\newcommand{\ta}{\tau}
\newcommand{\f}{\phi}

\newcommand{\F}{\Phi}

\newcommand{\om}{\omega}
\newcommand{\Om}{\Omega}
\newcommand{\Th}{\Theta}

\newcommand{\R}{{\mathbb R}}

\newcommand{\N}{{\mathbb{N}}}

\newcommand{\K}{{\mathbb{K}}}

\newcommand{\ab}{{\mathbf a}}

\newcommand{\cb}{{\mathbf c}}
\newcommand{\db}{{\mathbf d}}

\newcommand{\gb}{{\mathbf g}}
\newcommand{\hb}{{\mathbf h}}

\newcommand{\nb}{{\mathbf n}}

\newcommand{\qb}{{\mathbf q}}

\newcommand{\sbb}{{\mathbf s}}
\newcommand{\tb}{{\mathbf t}}

\newcommand{\xb}{{\mathbf x}}
\newcommand{\yb}{{\mathbf y}}

\newcommand{\Bb}{{\mathbf B}}

\newcommand{\Cb}{{\mathbf C}}

\newcommand{\Hb}{{\mathbf H}}

\newcommand{\Lb}{{\mathbf L}}
\newcommand{\Mb}{{\mathbf M}}
\newcommand{\Nb}{{\mathbf N}}

\newcommand{\Tb}{{\mathbf T}}

\newcommand{\Zb}{{\mathbf Z}}

\newcommand{\AF}{\mathfrak A}

\newcommand{\dF}{\mathfrak d}

\newcommand{\gF}{\mathfrak g}

\newcommand{\HF}{\mathfrak H}

\newcommand{\KF}{\mathfrak K}

\newcommand{\LF}{\mathfrak L}

\newcommand{\MF}{\mathfrak M}

\newcommand{\NF}{\mathfrak N}

\newcommand{\RF}{\mathfrak R}

\newcommand{\SF}{\mathfrak S}

\newcommand{\TF}{\mathfrak T}

\newcommand{\ZF}{\mathfrak Z}
\newcommand{\XF}{\mathfrak X}
\newcommand{\YF}{\mathfrak Y}

\newcommand{\Ac}{{\mathcal A}}

\newcommand{\Dc}{{\mathcal D}}
\newcommand{\Ec}{{\mathcal E}}

\newcommand{\Gc}{{\mathcal G}}
\newcommand{\Hc}{{\mathcal H}}

\newcommand{\Lc}{{\mathcal L}}
\renewcommand{\Mc}{{\mathcal M}}

\newcommand{\Nc}{{\mathcal N}}

\newcommand{\Rc}{{\mathcal R}}
\newcommand{\Sc}{{\mathcal S}}

\newcommand{\Uc}{{\mathcal U}}
\newcommand{\Wc}{{\mathcal W}} 
\newcommand{\Xc}{{\mathcal X}}
\newcommand{\Yc}{{\mathcal Y}}
\newcommand{\Zc}{{\mathcal Z}}

\newcommand{\Fs}{\sc\mbox{F}\hspace{1.0pt}}

\newcommand{\Vs}{\sc\mbox{V}\hspace{1.0pt}}

\newcommand{\sgn}{\hbox{{\rm sign}}\,}
\newcommand{\supp}{\hbox{{\rm supp}}\,}

\DeclareMathOperator{\re}{{\rm Re}\,}

\newcommand{\diam}{\operatorname{diam\,}}

\newcommand{\Tr}{\operatorname{Tr\,}}

\newcommand{\1}{\mathbf 1}

\newenvironment{dedication}{\itshape\center}{\par\medskip}
\newenvironment{acknowledgments}{\bigskip\small\noindent\textit{Acknowledgments.}}{\par}

\newtheorem{thm}{Theorem}[section]
\newtheorem{cor}[thm]{Corollary}
\newtheorem{lem}[thm]{Lemma}

\newtheorem{proposition}[thm]{Proposition}

\theoremstyle{definition}
\newtheorem{defin}[thm]{Definition}

\newtheorem*{rem}{Remark}

\numberwithin{equation}{section}

\usepackage[hyperfootnotes=false,colorlinks=true,allcolors=blue]{hyperref}

\begin{document}


\title[Singular measures and Connes integration ]{Eigenvalues of singular measures and Connes noncommutative integration}

\author{Grigori Rozenblum}
\address{Chalmers Univ. of Technology; The Euler International Mathematical Institute, and St.Petersburg State University}
\email{grigori@chalmers.se}

\begin{abstract} In a domain $\Om\subset \R^\Nb$ we consider compact, Birman-Schwinger type operators of the form $\Tb_{P,\AF}=\AF^* P\AF$ with $P$
 being a Borel measure in $\Om,$ containing a singular part and $\AF$ being an order $-\Nb/2$ pseudodifferential operator. For a class of such operators, we obtain a proper version of H.Weyl's law for eigenvalues,
  with order not depending on dimensional characteristics of the measure. These results lead to establishing measurability, in the sense of Dixmier - Connes, of such operators and the noncommutative version of integration over Lipschitz surfaces and rectifiable sets.
\end{abstract}
\maketitle
\begin{dedication}
In the memory of Misha Shubin, a friend and a great mathematician.
\end{dedication}


\section{Introduction}
\subsection{Operators associated with singular measures and their spectrum}
In the recent paper \cite{RSh2},  Birman-Schwinger (Cwikel) type operators  in a domain $\Om\subseteq\R^\Nb$ were considered, namely, the ones  having the form $\Tb_P=\AF^*P\AF$. Here  $\AF$ is a pseudodifferential operator in $\Om$ of order $-l=-\Nb/2$ and $P=V\m$ is a finite signed Borel measure containing a singular part. We found out there that for such operators, properly defined using quadratic forms, for a wide class of measures,  an  estimate for eigenvalues $\la^{\pm}_k=\la^{\pm}_k(\Tb_P)$ holds with order $\la^{\pm}_k=O(k^{-1})$ with coefficient involving an Orlicz norm of the weight function $V$. For a subclass of such measures, namely, for the ones whose singular part is a finite sum of measures absolutely continuous with respect to the surface measures on disjoint compact  Lipschitz surfaces of arbitrary dimension, an asymptotic, Weyl type,  formula for eigenvalues was proved, with all surfaces, independently of their dimension, making the same order contributions. In the present paper we discuss some generalizations of these results and their consequences for defining  noncommutative integration with respect to singular measures.

Our considerations are based upon the variational (via quadratic forms) approach to the spectral analysis of differential operators in a singular setting, in the form developed in 60-s and 70-s by M.Sh. Birman and M.Z. Solomyak. This approach enables one to obtain, for rather general spectral problems, eigenvalue estimates, sharp both in order and in the class of functional coefficients involved, this sharpness confirmed by exact asymptotic eigenvalue formulas. In the initial setting, this approach was applied to measures $P$ absolutely continuous with respect to the Lebesgue measure. Passing to singular measures, it was previously  found that, for the equation
 $-\la\D(X)=Pu(X), \, X\in\Om\subseteq \R^\Nb$, if the singular part of $P$ is concentrated on a smooth compact  surface inside $\Om$ (or on the boundary of $\Om$, provided the latter is smooth enough), it makes contribution of the order, different from the one produced by the absolutely continuous part, see, e.g., \cite{Agr} or \cite{Kozh}. It happens always, with the only exception of  the case $\Nb=2$, where the above orders are the same. For a class  of singular self-similar measures $P$, K.Naimark and  M.Solomyak established  in \cite{NaiSol} two-sided estimates for eigenvalues. And it turned out there that the \emph{order} of two-sided eigenvalue estimates  depends generally on the parameters used in the construction of the measure, in particular, on the Hausdorff dimension of its support.  However, in the single case, again of the dimension being equal to $2$, this dependence disappears, and the eigenvalues have one and the same order for all measures in the class under consideration, independently, in particular of their dimensional characteristics. In a very recent study in \cite{KarSh}, a new approach to the spectral problem $-\la\D u(X)=Pu(X), \, X\in\Om\subseteq \R^2$ have been developed, establishing, again, for a wide class of singular measures, upper eigenvalue estimates of one and the same order, independently on the Hausdorff dimension of the support of the measure. It became rather intriguing to understand which mechanism lies under this exceptional feature   of spectral problems in dimension 2.

In \cite{RSh2}  (main ideas and some results were announced in \cite{RSh1})  the above spectral problems have been generalized to an arbitrary dimension $\Nb$, so the eigenvalue properties were studied of  operator $\AF^* P\AF$, where $\AF$ is an order $-l=-\Nb/2$ pseudodifferential operator in a domain $\Om\subseteq \R^\Nb$ and $P$ is a signed measure of the form $V\m$, with Ahlfors regularity conditions imposed upon the measure $\m$ and  with  weight $V$ belonging to  a certain Orlicz class with respect to $\m$. Such a measure $\m$ is equivalent to the Hausdorff measure of some dimension $d, 0<d\le \Nb$ on the support of $\m$.  This operator is a natural generalization of the Birman-Schwinger operators which since long ago have been playing an important part in spectral and scattering theory. An eigenvalue estimate for this operator was established  in \cite{RSh2}, of order  depending neither on the dimension $\Nb$ nor on the Hausdorff dimension of the measure $\m$.  For $\m$ being the Hausdorff measure on a Lipschitz surface of any positive dimension and codimension, the asymptotics of eigenvalues has been found, supporting the sharpness of, more general, upper estimates.

In the present paper we extend results of  \cite{RSh2} to some wider class of measures and operators. These results lead to  establishing the noncommutative  measurability in the sense of A.Connes of the generalized Birman-Schwinger operators and in this way  we define the noncommutative  integral with respect to such measures, involving an analogy of the Wodzicki  residue. In particular, we present a noncommutative version of the Hausdorff measure for  a class of 'rectifiable',  in the sense of geometric measure theory, sets.

Our considerations have their roots in  results and constructions of the paper \cite{RSh2}, written jointly with Eugene Shargorodsky.   The author expresses deep gratitude to Eugene for benevolent attention and stimulating discussions. He also thanks Rapha$\ddot{e}$l Ponge for  explaining crucial facts about singular traces and the notion of measurability.


\subsection{Singular traces and Dixmier--Connes' integral}
In the huge and expanding field of the noncommutative geometry  (NCG) initiated by A. Connes, \cite{Connes0}, \cite{Connes1}, an important direction of studies deals with the notion of the noncommutative integral. Following the general  idea, in order to define integral on some algebra $\Ac$ of objects (say, functions), we associate, by means of some linear mapping $\f,$ with an object $\ab\in \Ac,$ a compact operator $\Tb=\f(\ab)$ belonging to the Dixmier-Matsaev ideal $\MF_{1,\infty}$ (consisting of operators $\Tb$ with singular numbers estimate $\sum_{k\le n}s_k(\Tb)=O(\log n).$) This ideal is larger than the trace class ideal $\SF^1$ and even larger than the ideal $\SF_{1,\infty}$ of operators with singular numbers satisfying the estimate $s_k(\Tb)=O(k^{-1}).$ On the ideal $\MF_{1,\infty}$, it is possible to define \emph{singular traces}, namely, continuous functionals  $\ta$ which are linear, positive, unitarily invariant, satisfy the trace property $\ta(\Tb_1\Tb_2)=\ta(\Tb_2\Tb_1)$ (provided the products here belong to $\MF_{1,\infty}$), and, finally, vanishing on the trace class ideal, see \cite{LSZ.Trace}. There are quite a lot of singular traces, the most important ones are obtained in the following way. For \emph{nonnegative} operators $\Tb\in\MF_{1,\infty}$, one considers the functional
\begin{equation}\label{weaklim}
\ta_0(\Tb)=\lim_{n\to\infty} (\log (n+2))^{-1} \sum_{k\le n}s_k(\Tb),
\end{equation}
on the subspace of  those operators $\Tb$ for which  this limit exists. This functional turns out to be linear on the cone of positive operators. After further extension by linearity, it is defined on a closed subspace in $\MF_{1,\infty}$ and continuous. Among Hahn-Banach continuous extensions $\ta$ of $\ta_0$ to the whole of $\MF_{1,\infty}$ there exist ones that satisfy the conditions in the definition of the singular trace. One can adopt $\ta(\f(\ab))$ as the integral of the object $\ab,$ and it is now universally called \emph{Dixmier--Connes' integral}. Having fixed such generalized trace for nonnegative operators, one can extend it by linearity to arbitrary operators in this class, since any compact operator is a linear combination of four nonnegative ones (there are certain limitations for this procedure, see, e.g., \cite{LSZ.Trace}.)

One of the earliest realizations of this scheme in \cite{Connes1} consists of recovering the integral of a function $V$ on a $\Nb-$ dimensional Riemannian manifold $\Mc$ by means of the singular trace of some  operator $T_V$ related with $V$. Initially it was proposed to consider operator  $T_V=V(-\D+1)^{-\Nb/2}$, where $\D$ is the Laplace-Beltrami operator on $\Mc$, for a smooth function $V$. It was established in \cite{Connes1} that
\begin{equation}\label{Weylw}
 \ta(T_V)=\pmb{\varpi}_{\Nb}\int V d\m_\Mc,\, \pmb{\varpi}_{\Nb}= \frac{\om_{\Nb-1} }{\Nb (2\pi)^\Nb}
\end{equation}
where $\m_\Mc$ is the Riemannian measure on $\Mc$ and $\om_{\Nb-1}$ is the measure of the unit sphere in $\R^\Nb$.
  In particular, it follows that this operator $T_V$ is \emph{measurable} in the sense that $\ta(T_V)$ has the same value for all positive normalized singular traces $\ta$. Such measurability, and even universal measurability (in the sense of \cite{Connes1}) results follow, in particular, from the fact that the limit in \eqref{weaklim} exists. Moreover, by  generalizations of Weyl's law, for $\Tb=T_V,$ even the limit
\begin{equation}\label{limstrong}
    \lim_{k\to\infty}{k s_k(\Tb)},
\end{equation}
exists. Of course, the existence of the limit in \eqref{limstrong} implies such existence for the limit in \eqref{weaklim}, but the converse is not, generally, correct. The result on measurability was established also for a noncompact $\Mc,$ namely, for $\Mc=\R^\Nb$, under the condition that $V$ has compact support.

In the process of time, this procedure of construction of singular traces and the notion of measurability developed quite noticeably. Quite a few versions of singular traces appeared since \cite{Connes0}, \cite{Connes1}, differing by the properties of the class of initial traces to be extended and by the ideals to which this trace is extended; reviews of these versions can be found in \cite{LSZ.Trace}, \cite{SSU}, \cite{SSUZ}, \cite{SU}, and, especially, in \cite{LoSuMeas} where an extensive hierarchy of classes of traces and corresponding   nonequivalent notions of operator measurability has been described. Our analysis does not distinguish between these versions, therefore we will refer to Dixmier--Connes' measurable operators and Dixmier--Connes' integral.

 In the latest decennium, quite an activity developed,
concerning  extending  these results to less regular functions $V,$ see, e.g., \cite{KLPS}, \cite{LPSMeasure}, \cite{LoSuMeas}, \cite{LSZ.Trace}, \cite{LSZ}. Say, if $V$ belongs to $L_2$ and, in the case of $\Mc=\R^\Nb$, has a compact support,  operator $T_V$ belongs to $\SF_{1,\infty},$ is measurable, and the usual expression \eqref{Weylw} is valid for the singular trace. However (and this was noticed, e.g., in \cite{LPSMeasure}),  if $V$ is outside $L_2(\Mc)$,  operator $T_V$ may turn out to be not bounded, to say nothing of being compact. Therefore, a proposal was made in \cite{LPSMeasure} to consider a different, 'symmetrized,' operator associated with $V,$ namely, $\Tb_V=(-\D+1)^{-\Nb/4}V(-\D+1)^{-\Nb/4}.$ Being properly defined, this operator is bounded and even compact for $V\in L_p,$ $p>1$, with compact support, self-adjoint for real-valued $V$, and the trace formula \eqref{Weylw} holds. Much more hard is the   case $p=1$: here the right-hand side in \eqref{Weylw} is still finite but the question about the existence and the value of the trace on the left-hand side turns out to be rather complicated. Simple examples show that for a general $V\in L_1$,  operator $\Tb_V$ may fail to be bounded, moreover this effect may be caused both by  local singularities of $V$ and by an insufficiently fast decay of $V$ at infinity (for a noncompact $\Mc$). Very recently the conditions, rather sharp, were elaborated granting the compactness of $\Tb_V$ as well as its membership in $\MF_{1,\infty}$ and the validity of the integration formula. These conditions require $V$ to be just a little bit better  than simply lying in  $L_1$, namely, to belong to a certain  Marcinkiewicz space, see \cite{SZSol}, where $\Tb_V$ was called \emph{the Cwikel operator}.
\subsection{Birman-Schwinger operators and their eigenvalues}
Independently of these results, and even considerably earlier, spectral properties of operators of the form $\Tb_V$ have been the object of intensive studies by specialists in mathematical physics. The case of the highest interest was the one of $\Mc=\R^\Nb$; here this topic is closely related with the eigenvalue analysis of the Schr\"odinger operator. We define, for a compact self-adjoint operator $\Tb$, $n_{\pm}(\la, \Tb)$ to denote the number of eigenvalues of $\pm\Tb$ in $(\la,\infty).$ Operators like  $\Tb_V$ are called Birman-Schwinger operators, and by the \emph{Birman-Schwinger principle},
\begin{equation}\label{BirSchPrin}
    n_{+}(\la, (-\Delta+E)^{-q/2}V(-\Delta+E)^{-1/2})=N_-((-\Delta+E)-\la^{-1}V),\, \la>0,
\end{equation}
where  the expression on the right is the number of negative eigenvalues of the Schr\"odinger  operator. In dimension $\Nb>2,$    equality \eqref{BirSchPrin} is valid for $E=0$ as well, and  sharp results on the eigenvalue estimates and asymptotics have been obtained quite long ago. However, in dimension $\Nb=2$ some deep modifications are needed in the expression  on the left-hand side for the proper version of \eqref{BirSchPrin} to hold.  Anyway, for $V\in L_p(\R^2), \, p>1$ with compact support, estimates and asymptotics of eigenvalues of $\Tb_V$ were known as long ago as in 1972, see \cite{BirSolLect} and references therein. Sharper results, approaching $p=1,$ were obtained by M.Z.Solomyak \cite{Sol94} in 1994, where the condition on $V$, besides the compactness of support, involved the membership of $V$ to certain Orlicz class, i.e., again, a little bit better than $V\in L_1$. In the same paper, the case of any even dimension $\Nb$ was handled in a similar way. Problems without compact support condition were studied in \cite{BLS} and further on, see the latest developments in \cite{Shar2D}.

These two lines of  study converged  recently in the paper \cite{SZSol}, where the method of piecewise polynomial approximations, in the version elaborated by M.Z. Solomyak, was adapted to prove the measurability of the  operator $\Tb_V$ on $\R^{\Nb}$ and on the torus $\mathbb{T}^{\Nb}$ with $V$ in the Marcinkiewicz class.

Having in mind an extension of the notion of Connes' integral, we take a somewhat different point of view. We are looking for defining integration of \emph{measures}, including singular ones, in the context of the noncommutative geometry. The starting point will be a re-statement of the above results for a  measure $P$ containing, possibly, a singular component. For a, possibly, unbounded, domain $\Om\subset \R^{\Nb}$, we consider an operator $\AF$ in $L_2(\Om)$. It is a  pseudodifferential operator of order $-l=-\Nb/2,$ acting as $\AF:C_0^{\infty}(\Om)\to C_0^{\infty}(\Om)$ (we call such operators compactly supported). In the leading example,  the principal symbol of $\AF$ equals $a_{-l}(X,\X)=|\X|^{-l}$ for $X$ in a proper bounded subdomain $\Om'\subset\overline{\Om'}\subset\Om$. As examples of such $\AF$ may serve $\theta(X)\LF \theta(X)$, where $\theta(X)$ is a smooth function in $C^{\infty}_{0}(\Om)$, which equals $1$ on $\Om'$, and $\LF$ may be the inverse of the proper power of the Laplacian on $\Om$ with some elliptic (e.g., Dirichlet) boundary conditions or the operator $(-\Delta+1)^{-\Nb/4}$ in $\R^\Nb$. For the, probably, most interesting, case $\Om=\R^{\Nb},$ we consider $\AF=\AF_0\equiv(1-\Delta)^{-\Nb/4}.$
\subsection{Birman-Schwinger operators for singular measures}
Let $P$ be a signed Radon measure on $\Om$. With such measure and operator  we associate the  Birman-Schwinger (or Cwikel) operator in the following way. If $P$ is absolutely continuous with respect to the Lebesgue measure, with density $V(X),$ $P=V(X)dX$, we set
\begin{equation}\label{BScOPerator.AC}
    \Tb_P\equiv\Tb_{V}\equiv \AF^* P\AF\equiv \AF^* V\AF.
\end{equation}
If the function $V$ is bounded, $\Tb_V$ is automatically a bounded operator. Some more trouble arises if $V$ is an unbounded function. An approach to defining this operator was (for $\AF=(-\D+1)^{-\Nb/4}$) proposed in \cite{LPSMeasure}, based upon tracing between which Sobolev spaces separate factors in \eqref{BScOPerator.AC} act. We use a different approach, equivalent to this one for absolutely continuous measures, however allowing extension to measures in more general classes. Namely, we associate with \eqref{BScOPerator.AC} the quadratic form in $L^2(\Om),$
\begin{equation}\label{quadr.form V}
    \tb_{P,\AF}[u]=\int_\Om|(\AF u)(X)|^2 P(dX)=\int_\Om|(\AF u)(X)|^2 V(X)dX.
\end{equation}
If this quadratic form is well-defined and bounded in $L_2(\Om)$, it defines there an operator, which we will accept as $\Tb_{P,\AF}$. In particular, if $V\in L_\infty$, this operator, obviously, coincides with  $\AF^* V\AF$ understood as a product of three bounded operators. Moreover, if we set $V=|V|\sgn V=U^2\sgn V$, it follows from \eqref{quadr.form V} that

\begin{equation*}
    \tb_{P,\AF}[u]=\int_\Om \overline{(\AF u)(X)}(\AF u)(X) V(X)dX  =\langle(U\AF)^*(\sgn V) (U\AF)u,u\rangle_{L_2(\Om)},
\end{equation*}
therefore
 \begin{equation}\label{factorization}
 \Tb_P=(U\AF)^*(\sgn V)(U\AF).
 \end{equation}
This representation has been also used in \cite{LPSMeasure} and \cite{LSZ} for $V$ in Marcinkiewicz and Orlicz classes.

We are interested in expanding the definition \eqref{BScOPerator.AC} to the case when  measure $P$ contains a singular part, $P=P_{ac}+P_{sing}$. Namely, we set, with this new meaning,
\begin{equation}\label{quadrFormSing}
    \tb_{P,\AF}[u]=\int_\Om|(\AF u)(X)|^2 P(dX).
\end{equation}
This quadratic form is defined initially on smooth functions in $L_2(\Om).$ If it proves to be bounded in $L_2(\Om)$ -- and we will find sufficient conditions for this boundedness (see Section \ref{Sect.Est}) --  it can be extended to the whole $L_2(\Om)$ by continuity and  we accept the corresponding bounded self-adjoint operator as $\Tb_{V}\equiv \AF^* P \AF$. Generalizing the case of an absolutely continuous measure, this operator  admits a  factorization similar to  \eqref{factorization}.

For an unbounded domain $\Om,$ especially, for $\Om=\R^\Nb$, we always suppose here that  measure $P$ has compact support. It is well known, even for an absolutely continuous measure, that for the whole $\R^\Nb$ the behavior of $P$ at infinity requires  rather special considerations since infinity can make  to the eigenvalue counting function a stronger contribution than the local terms, see \cite{BLS}. Even on the plane, $\Nb=2,$ sharp conditions for the Birman-Schwinger operator to satisfy the Weyl formula,
 are still unknown up to now, the best results being obtained in \cite{Shar2D}. Being interested in local effects, we do not touch upon such problems here.

Let now $\Mc$ be a compact $\Nb$-dimensional  Riemannian manifold, with  Riemannian measure $\pmb{\m}_{\Mc}$; we denote by $\D_{\Mc}$ the corresponding Laplace-Beltrami operator, self-adjoint in $L_2(\Mc,\pmb{\m}_{\Mc})$. For a finite signed Borel measure $P$ on $\Mc$, we define the operator $\Tb_P=\Tb_{P,\Mc}$ in $L_2(\Mc,\pmb{\m}_{\Mc})$ by means of the quadratic form
\begin{equation}\label{quadrFormSing1}
    \tb_{P,\Mc}[u]=\int_\Mc|((-\D_{\Mc}+1)^{-\Nb/4} u)(X)|^2 P(dX),\,  u\in L_2.
\end{equation}
Again, if  measure $P$ is absolutely continuous with respect to the Riemannian measure $\pmb{\m}_{\Mc}$, $P=V \pmb{\m}_{\Mc}$, the operator $\Tb_{P,\Mc}$ coincides with the properly defined operator
 $(-\D_{\Mc}+1)^{-\Nb/4}V(-\D_{\Mc}+1)^{-\Nb/4},$ and the integrability results in \cite{LPSMeasure} and \cite{LSZ} apply.
 If $\Uc$ is a local co-ordinate neighborhood in $\Om$, containing the support of the measure $P$,  with the diffeomorphism $\Fs: \Uc\to \Wc\subset \R^{\Nb}$ then
 the operator $\Tb_{P,\Mc},$ by usual localization,  transforms to an operator of the type  $\AF^{{}^*} (\Fs^{{}^*}P)\AF,$ where $\Fs^{{}^*}P$ is the measure in $\Wc,$ the pull-back  of $P$ under the mapping $\Fs$, and $\AF$ is the order $-\Nb/2$ pseudodifferential operator in $\Wc$, actually $(-\D_{\Mc}+1)^{-\Nb/4}$ expressed in  local co-ordinates in $\Wc$.
 So, as soon as such localization is justified (this is done in a rather traditional straightforward way) we are left with the task of spectral analysis of the operator $\Tb=\Tb_{P,\AF}$ in a domain in $\R^\Nb$. In fact, without additional work, we can consider in this way more general operators on manifolds, the ones having the form $\AF^* P\AF$ where $\AF$ is an order $-l=-\Nb/2$ pseudodifferential operator on $\Mc$, with the result having a similar form.
\subsection{Main results}
  Our aim is two-fold. First, to extend the class of measures and operators for which $\Tb_{V,\AF}$ belongs to the class $\MF_{1,\infty}$, so that $\ta(\Tb_{P,\AF})$ is finite (but, probably, depends on the choice of the singular trace $\ta$). This property follows from the eigenvalue estimates for $\Tb_{V,\AF}$. Such estimates have been obtained in \cite{RSh2} but we need  a somewhat more general class of measures, however the reasoning is rather similar. Secondly, we are going to find a subclass of  singular measures for which operator $\Tb_{P,\AF}$  is measurable, i.e., this trace does not depend on the above choice. In our case, this measurability follows first by establishing the Weyl asymptotic formula for eigenvalues of $\Tb_{P,\AF}$ for a measure concentrated on a Lipschitz surface $\Si$,
    \begin{equation}\label{As.Si}
    \lim_{\la\to 0}\la n_{\pm}(\la, \Tb_{P,\AF})=\int_\Si \ro_{\AF}(X)V_{\pm}(X)d\m_\Si,
  \end{equation}
 where  density $\ro_{\AF}(X)$   is determined by  operator $\AF.$ We find an explicit expression for this density; in the leading case $\AF=\AF_0=(1-\D)^{-\Nb/2},$ $\Hb=\Zb(d,\dF)$ is a constant depending on the dimension $d$ and codimension $\dF$ of the surface $\Si$.
 It follows from \eqref{As.Si} that operator $\Tb_{P,\AF}$ is measurable with the expression for the singular trace
 \begin{equation}\label{Connes.Surf}
    \ta(\Tb_{P,\AF})=\int_{\Si} \ro_{\AF}(X)V(X)d\m_\Si.
 \end{equation}
  We combine our eigenvalue estimates and asymptotics for measures on Lipschitz surfaces to extend them further to measures having more complicated structure, see Section \ref{Rect.Sect}. We prove   extensions of the Weyl formula (for the case of $\AF=\AF_0$ only, since it is  too cumbersome for the general case) for  a wide class of measures of the form $P=V \Hc^d$ (the latter symbol denotes the Hausdorff measure of dimension $d$) supported on a   \emph{rectifiable} set $\XF$, in the sense of geometric measure theory. These are sets in $\R^{\Nb}$ of Hausdorff dimension $d,$ $0<d<\Nb$ that can be covered by a countable collection of Lipschitz surfaces of dimension $d$. Some effective criteria for a set to be rectifiable exist, expressed in terms of local densities. We prove the Weyl formula
 \begin{equation}\label{As.Rect}
    \lim_{\la\to 0}\la n_{\pm}(\la, \Tb_{P,\AF_0})=\Zb(d,\dF)\int_{\XF} V_{\pm}(X)\Hc^d(dX),
  \end{equation}
  with coefficient $\Zb(d,\dF)$ explicitly written and depending only on the dimension $d$ and codimension $\dF=\Nb-d$ of the set
  This formula again, implies   Connes' measurability of  operator $\Tb_{P,\AF_0}$ in this setting and the trace formula,
  \begin{equation}\label{Connes.Rect}
    \ta(\Tb_{P,\AF_0})=\Zb(d,\dF)\int_{\XF}V(X)\Hc^d(dX).
 \end{equation}
  Singular trace formulas \eqref{Connes.Surf},  \eqref{Connes.Rect} can be understood as generalizations of the Wodzicki residue to a rather singular setting. The least restrictive are conditions for the case of the support $\XF$ of measure  $\m$ having dimension one. Here, the eigenvalue asymptotics and trace formula are justified as soon as the measure $\m$ is 1-Ahlfors regular, while its support is  connected (even more general conditions exist but they are somewhat cumbersome, see in the paper later.)

  Here we would like to note that, traditionally in the NCG community, Weyl type asymptotic eigenvalue formulas were not being used for proving measurability of integrals and  for calculating the singular trace. Even in a very recent review paper \cite{LSZ.advances}, it is written that the calculation of singular traces is rarely done using the explicit eigenvalue asymptotics. Most cases in Noncommutative Geometry use the Zeta-function approach and the heat equation approach  to calculate the singular trace. In fact, even in the review \cite{LSZ.advances}, published in 2019, papers by M.Sh.Birman and M.Z.Solomyak, where eigenvalue asymptotics for rather general pseudodifferential operators, as well as later developments in this topic, have not even been put on the reference list. In the present paper, in the opposite, the measurability and singular trace results are easy consequences of our Weyl type asymptotic eigenvalue formulas.

 The proofs of the above asymptotic and trace formulas are presented in the paper further on. It turns out that, in the most general setting, due to the linearity property of the singular trace,   the proof of the trace formula for measures of complicated structure \eqref{Connes.Rect} is considerably more elementary than the one of eigenvalue asymptotics \eqref{As.Rect}, although the former follows also immediately  from the latter one. For Readers interested in trace formulas only, we present independent, rather short, proofs as well.

 In cases when we are unable to prove asymptotic formulas for eigenvalues (in particular, for a fractional Hausdorff dimension of the support of the measure), we can, nevertheless, show that our upper eigenvalue estimates are sharp in order and are close to be sharp in the class of measures, by means of finding lower eigenvalue estimates of the same order. The author believes strongly that  Connes' measurability  holds in these cases as well, in particular for measures having fractional Hausdorff dimension, but there are no visible approaches to this problem at the moment.

\subsection{Alternative approaches}
Finally, in this section, we note that a quite different approach to the Dixmier-Connes integral against singular measures possessing some regular fractal structure has been developed some time ago by M. Lapidus, J.Fleckinger and their co-operators, see \cite{Lap1997}, \cite{Lap2008} and  references therein. For a fractal measure on a set $\Xc$ in the Euclidean space, operators were considered, using the Laplacian on this fractal set $\Xc$ itself, and singular traces of these operator were investigated. The required power of this fractal Laplacian used in this construction depends on the fractal dimension of the support
 of the measure, while our construction uses one and the same operator $\AF$ for all admissible measures. It might be interesting to find a connection between these two approaches.

  The setting by D.Edmunds and H.Trieblel,  see \cite{ET}, \cite{Triebel1}, \cite{Triebel2}, where an operator is associated with a rather general, Ahlfors regular, singular measure,  is closer to ours. However, the eigenvalue estimates obtained by these authors are usually not sharp in order and/or in the class of the weight functions (this is stressed, e.g., in Discussion 27.3,  Remark 27.5, or Remark 27.10  in  \cite{Triebel2}). Such circumstance prevents one from deriving eigenvalue asymptotics for operators under consideration -- this task being the main topic of our paper. Note however, that the crucial fact in our setting, namely, that for the case of the order of the operator being equal to the half of the dimension of the space, the eigenvalues decay order does not depend on this dimension, and on the dimensional characteristics of the measure either, has been predicted - and in some cases discovered  -- by the authors of these books, see, e.g.,  Proposition 28.10 in \cite{Triebel2}.

\section{Boundedness}\label{Bddness.Sect}
First, we consider operators $\AF$ having compact support, $\AF:C_0^\infty(\Om)\to C_0^\infty(\Om),$ $\Om\subset\R^{\Nb}.$
We set here more concrete conditions for the measure $P$ to define a bounded quadratic form   \eqref{quadr.form V} and further a bounded operator $\Tb_{V,\AF}$. Note, however, that due to localization properties, the choice of the set $\Om$ as well as the choice of the operator $\AF$ outside a neighborhood of the support of measure $P$ have a weak influence on the spectral properties of the operator $\AF,$ see Proposition \ref{Prop.Local}. We will use this freedom systematically.

 Since $\AF$ is a pseudodifferential operator of order $-l,$ it is sufficient to find conditions for the boundedness of the quadratic form $\sbb[v]=\int_{\Om}|v(X)|^2P(dX)$ in the Sobolev space $H^{l}_{0}(\Om)$. Smooth functions are dense in $H^{l}_{0}(\Om),$ therefore it suffices to justify the inequality
\begin{equation}\label{MazIneq}
    \left|\int_{\Om}|v(X)|^2P(dX)\right|\le C(V,\m) \|v\|^2_{H^l}(\Om), \,
\end{equation}
for $ v\in C^{\infty}_0(\Om)$ and then extend to the whole of $H^{l}_{0}(\Om)$ by continuity.
Here $\|v\|_{H^l}(\Om)$ is the usual norm in the Sobolev space;
\begin{equation}\label{Sob}
\|v\|^2_{H^l}(\Om)=\int_{\Om}|v|^2 dX+\|v\|^{ 2}_{(hom),H^l}(\Om),
\end{equation}
where the homogeneous seminorm in \eqref{Sob} is defined as $\|v\|^{ 2}_{(hom),H^l}(\Om)=\int_{\Om}|\nabla^l v|^2dX$ for an integer $l=\Nb/2$ and
\begin{equation*}
    \|v\|^{ 2}_{(hom),H^l}(\Om)=\int_{\Om\times\Om}\frac{|\nabla^{l-\frac12}v(X)-\nabla^{l-\frac12}v(Y)|^2}{|X-Y|^{\Nb+1}}dXdY
\end{equation*}
for a half-integer $l$.

Basic results in this direction have been established in works by V.Maz'ya. A sufficient condition, being also a necessary one for a positive measure $P$, is given by Theorem 11.3 in \cite{MazBook} in terms of capacity (for $l=1$, sharp conditions have been found  even for a signed measure $P$, however, for larger $l,$ such conditions seem to be still unknown.) We are interested in conditions expressed in more elementary terms, and therefore we use  Theorem 11.8 and Corollary 11.8/2 in \cite{MazBook}.

Measures $P$ considered here have the form $P=V\mu$, where $\m$ is some fixed singular measure, and $V$ is a $\m$-measurable real function which we call 'density'; our results consist of describing classes of densities for a given  $\m$ for which the required estimates for the operator norm, resp., eigenvalues, hold.  So, let $\m$ be a finite Borel measure on $\Om$. We denote by $\Mb=\Mb(\m)$ its support, the smallest closed set of full measure; we always suppose that $\Mb$ is compact. We do not usually distinguish between measure $\m$ considered on $\Mb$ and its natural extension by zero to the whole of $\Om:$  $\m(E):=\m(E\cap\Mb).$

Conditions imposed on density $V$ are expressed in terms of Orlicz spaces.
These spaces  have been long ago found to be  the proper instrument in the treatment of the critical case $2l=\Nb$. For a detailed exposition of these spaces, see, e.g., \cite{Krasn.Orlicz} or \cite{Rao}. We use  a special choice of Orlicz functions.
The Orlicz space $L^{\Psi,\m}, $ $\Psi(t)=(1+t)\log(1+t)-t,$ consists of $\m$-measurable functions $V$ on $\Mb$, satisfying $\int_{\Mb}\Psi(|V(X)|)d\m(X)<\infty$. For a subset $E\subset\R^{\Nb}$,  the norm in $L^{\Psi,\m}$ is defined by

\begin{equation}\label{exp.norm}
 \left|\left|V\right|\right|_{L^{\Psi,\m}(E)}=\inf\{\vs:\int_E\Psi(|V|/\vs)d\m\le 1; \, \m(E)>0\}.
\end{equation}

Function $\Phi(t)= e^t - 1 - t$ is Orlicz dual to $\Psi.$ Thus, the Orlicz space $L^{\Phi,\m}$ consists of functions  $g$ satisfying $\int_\Mb \Phi(|g|/\vs) d\m<\infty$ for some $\vs>0$ with norm defined similarly to \eqref{exp.norm}
\begin{equation*}
    \|g\|_{L^{\Phi,\m}(E)}=\inf\{\vs:\int_E\Phi(|g|/\vs)d\m\le 1\}.
\end{equation*}
Measure $\m$ may be omitted in this notation, as long as this does not cause a misunderstanding.
So,  functions in $L^{\Psi}$ are a tiny little bit better than just lying in $L_1(\m)$, while functions in $L^{\Phi}$ may be unbounded, but only very weakly.

By known embedding properties of Sobolev spaces, as soon as  measure $\m$ possesses at least one point mass, the corresponding quadratic form
$\sbb[v]$ is not bounded in $H^l$, in other words, functions in $H^l,\, l=\Nb/2,$ are not necessarily continuous or even essentially bounded. However, their possible unboundedness is very weak: they belong to $L^{\Psi}.$

The boundedness, to be used later on,  of the quadratic form $\sbb[v]$ in $H_0^l(\Om)$ (or $H^{l}(\Om)$) follows from two facts. One of them is the general H\"older type inequality (see, e.g., \cite{Krasn.Orlicz}) for Orlicz spaces,   having, in our case, the form
\begin{equation}\label{Holder}
    \left|\int |v|^2 V d\mu\right|\le C \left|\left| v^2\right|\right|_{L^\Phi}\|V\|_{L^\Psi};
\end{equation}
(the constant $C$ here is an absolute one; it would equal $1$ if we have used some other, equivalent, norms in the Orlicz spaces.)

Another ingredient is Corollary 11.8/2
in \cite{MazBook}. In our case, for $p=2$, $l=\Nb/2$, it sounds:
\begin{lem}\label{LemMaz}  The estimate
  \begin{equation}\label{MazIneq1}
    \left|\!\left|v^2\right|\!\right|_{L^{\Phi,\mu}}\le A \|v\|^2_{{H}^l(\Om)}
 \end{equation}
 holds for all $v\in H^l(\Om)$
 if and only if for some $\be>0$  measure $\m$ satisfies the inequality
 \begin{equation}\label{MazMeasCond}
    \m(B(r,X))\le C({\m}) r^{\be}, \, r<1, \, B(r,X):=\{Y: |Y-X|\le r\},
 \end{equation}
 for all $X\in \Mb$, with constant $A=A(\m)$ depending only on $\be$ and $C(\m)$ in \eqref{MazMeasCond}.
\end{lem}

Now we can formulate the required  boundedness condition which follows immediately by combining \eqref{MazIneq1} and \eqref{Holder}.
\begin{proposition}\label{Th.boundedness} Let measure $P$ have the form $P=V\m,$ where $V$ is a real $\m$-measurable function on the support of $\m$. Suppose that  $\m$ satisfies condition \eqref{MazMeasCond} and $V\in L^{\Psi}$. Then the inequality
\begin{equation*}
     \left|\int |v|^2 V d\mu\right|\le CA(\m)\|v\|^2_{H^l(\Om)}\|V\|_{L^{\Psi}}
\end{equation*}
is satisfied for all $v\in {H}^l(\Om)$ with constant not depending on $V$.
\end{proposition}
We return to operator $\Tb_{P,\AF}$ to obtain the boundedness condition.
\begin{thm} Let measure $\m$ satisfy \eqref{MazMeasCond}.
Then for any $V\in L^{\Psi},$ $P=V\m,$ operator $\Tb_{P,\AF}$ is bounded in $L_2(\Om)$ and
\begin{equation}\label{bddness}
    \|\Tb_{P,\AF}\|\le C(\AF)A(\m)\|V\|_{L^\Psi}.
\end{equation}
The constant $C(\AF)$ in \eqref{bddness} depends on operator $\AF$, and dimension $\Nb$, but not on the density   $V$.
\end{thm}

We can now present a description  of the action of the operator $\Tb_{P,\AF},$
 similar to the one given, for an absolutely continuous measure $P$, in \cite{LSZ}.  If  inequality \eqref{MazIneq} is satisfied for all $v\in {H^l_0}$, it follows, by the usual polarization,
  that $$\left|\int_{\Om}w(X)\bar{v}(X)P(dX)\right|\le C \|v\|_{H^l}\|w\|_{H^l}, \, v,w\in H^l_0.$$ The latter  inequality means that for a fixed $v$, the integral on the left is a continuous functional of the function $w$ in $H^{l}_0$, therefore $\bar{v}(X)P\in H^{-l}(\Om)$ for $v=\AF u$, $u\in L_2(\Om)$.
 Consequently, the result of application of the order $-l$ operator $\AF^*$ to $\bar{v}(X)P$ belongs to $L_2$, and so
 the operator defined by the quadratic form $\tb_{P,\AF}$ in $L_2$ factorizes as a composition of bounded operators,

 \begin{equation}\label{Composition}
    \Tb_{P,\AF}: L_2(\Om)\overset{\AF}{\longrightarrow}H^{l}_0(\Om)\overset{P}{\longrightarrow}H^{-l}(\Om)\overset{\AF^{*}}{\longrightarrow}L_2(\Om).
 \end{equation}

This representation is a natural generalization of the one used, e.g., in \cite{LSZ}, \cite{SZSol},  however it is less convenient than \eqref{bddness} when establishing norm and eigenvalue estimates.

\section{Eigenvalue estimates}\label{Sect.Est}
In order to obtain eigenvalue estimates for operator $\Tb_{P,\AF},$ we need to impose additional assumptions on the measure $\m$.
\begin{defin} A Radon measure $\m$ on $\R^\Nb$ with compact support $\Mb$ is called Ahlfors $s$-regular, $s>0$, if for some $\Cb>0$ and any $X\in\Mb,$
\begin{equation}\label{Aregular}
    \Cb r^{s}\le \m(B(X, r))\le \Cb^{-1} r^{s},\, r\le\diam \Mb
\end{equation}
for all $X\in\Mb.$
\end{defin}

Such measure is equivalent to the $s$-dimensional Hausdorff measure $\Hc^s$ (see, e.g., \cite{DavSem}, Lemma 1.2) on the support of $\m$. Note that $s$-regular measures satisfy condition \eqref{MazMeasCond} with $\be=s$.

In the Orlicz space $L^{\Psi,\m}$, for a Borel set $E,$ we introduce the  norm,
 \begin{equation}\label{AvNorm}
   \|V\|_{E}^{(av,\Psi;\m)}=\sup\left\{\left|\int_{E\cap\Mb}V d\m\right|:\int\F(|g|)d\m\le\m(E\cap\Mb)\right\},
 \end{equation}
 if $\m(E\cap\Mb)>0,$ and $\|V\|_{E}^{(av,\Psi)}=0$ otherwise.
Such \emph{averaged} norms have been first introduced by M.Z. Solomyak in \cite{Sol94} and were being used since then  in the study of the eigenvalue distribution in the critical case. The norm \eqref{AvNorm} is equivalent to the standard norm in $L^{\Psi,\m}$ but the coefficient in the equivalence depends on the the measure $\mu$
(in fact, on $\m(E)$).
Our basic result on the eigenvalue estimates  is the following:
\begin{thm}\label{EstTheor} Let  measure $\m$ with compact support satisfy  condition \eqref{Aregular} with some $\a>0$ and let $V\in L^\Psi$. Let $\AF$ be an order $-l=-\Nb/2$ pseudodifferential operator with compact support. Then for the operator $\Tb_{V\m,\AF}$
the following eigenvalue estimate holds
\begin{equation}\label{BasicEstimate1}
n_{\pm}(\la,\Tb_{V\m,\AF})\le C(\m) C(\AF)\la^{-1}\|V_{\pm}\|_{\supp \m}^{(av,\Psi;\m)}.
\end{equation}
\end{thm}

The proof is presented in detail in \cite{RSh2}. We note here only that it follows the pattern of the two-dimensional reasoning in \cite{KarSh}. In its turn, this  variational proof is based upon ideas used for obtaining a similar estimate in \cite{Sol94}, where an absolutely continuous measure $\m$ was considered. All of them have, as their starting point, the original  proof  of the CLR estimate, as presented in \cite{BirSolLect}.

It is convenient to eliminate further on the dependence of results on the domain $\Om$ and of the operator $\AF$ outside a neighborhood of the support of $\m$. This is done by means of the following estimate.
\begin{proposition}\label{Prop.Local}Let $\AF_1,$ $\AF_2$ be two   order $-l=-\Nb/2$ pseudodifferential operator in a bounded domain $\Om\in \R^\Nb$ with compact support such that $\AF_1 f=\AF_2 f$ in a neighborhood of $\Mb$  for $f$ supported in $\Om'$  and $\m$ be a finite Borel measure with support inside $\Om',$ $V\in L^{\Psi,\m}$ then

\begin{equation}\label{comparison}
    n_{\pm}(\la, \Tb_{P,\AF_1})-n_{\pm}(\la, \Tb_{P,\AF_2} )=o(\la^{-1}), \la\to 0.
\end{equation}
\end{proposition}
\begin{proof} Consider a cut-off function $\chi$ which equals $1$ in a sufficiently small neighborhood of $\Mb$ and equals zero outside another small neighborhood, so that operators $\chi\AF_1\chi$ and $\chi\AF_2\chi$ coincide.
The quadratic form of $\Tb_{P,\AF_j}$ is represented as
\begin{gather}\label{splitting1}
    \tb_{P,\AF_j}[u]=\int |\AF_j u|^2 P(dX)=\int |\chi^2\AF_j u|^2 P(dX)=\\\nonumber
    \int |\chi\AF_j\chi u|^2 P(dX)+\int 2 \re((-\chi[\AF_j,\chi] u)\overline{(\chi^2 \AF_j u)})P(dX)+ \int|\chi [\AF_j,\chi] u|^2 P(dX).
\end{gather}
In \eqref{splitting1}, the first term is the same for $\AF_1,\AF_2$, while the remaining terms contain commutators of $\chi$ with $\AF_j$, which are pseudodifferential operators of order $-l-1$. Quadratic forms with such operators in $L_2$, or, what is equivalent, quadratic forms $\int|v|^2 P(dX)$ in $H^s$, $s>\Nb/2$, have singular values decaying faster that $k^{-1},$ e.g., by Theorem 3.1 in \cite{BS}. Thus, operator $\Tb_{P,\AF_1},$ $\Tb_{P,\AF_2}$ differ by an operator with fast decaying singular values, and \eqref{comparison} follows from the Ky Fan inequalities.
\end{proof}
This property shows that the behavior of eigenvalues of our operators for a singular measure is determined by the operator $\AF$ restricted to arbitrarily small neighborhood of the support of the measure. Additionally, it enables localization of  operators, when considering  measures on manifolds. The same reasoning grants this kind of localization for the case when $\Om=\R^\Nb$ and $\AF=\AF_0=(1-\Delta)^{-l/2}.$

In Theorem \ref{EstTheor} and its consequences, it is important that  measure $\m$ has compact support. It is known since long ago that even for  $\m$ being the Lebesgue measure on $\R^\Nb,$ behavior of $V$ at infinity requires additional considerations (see, especially \cite{BLS} and \cite{KarSh}) and the contribution of infinity to the eigenvalue estimates may be stronger than the local one in \eqref{BasicEstimate1}. It this paper we are interested in local effects.

The eigenvalue estimate \eqref{BasicEstimate1} extends immediately by means of the Ky Fan inequality  to finite sums of  measures $P=\sum P_j$, $P_j=V_j\m_j$, where measures $\m_j$ may have different dimensions, e.g., satisfy \eqref{Aregular} with different values of $s$, including $s=\Nb$, the latter case corresponds to an absolutely continuous measure. However, the control over the constants in the estimates becomes rather cumbersome since the triangle inequality fails for the ideal $\SF_{1,\infty}$.

It follows from Theorem \ref{EstTheor} that operator $\Tb_{P,\AF}$ belongs to the ideal $\MF_{\1,\infty}$ and therefore singular Dixmier traces exist for $\Tb_{P,\AF}$.
We may not, however, declare at the moment that the operator $\Tb_{P,\AF}$ is measurable, without additional conditions imposed.

Results on eigenvalue estimates are easily carried over, by means of the same localization, to  spectral problems considered on closed manifolds.
 \begin{cor}\label{Cor.Man.Est} Let $\Mc$ be an $\Nb$-dimensional closed Riemannian manifold and $\m$ be a Borel measure, Ahlfors $s$-regular for some $s>0$. Let $V$ be a $\m$ - measurable real  function belonging to $L^{\Psi,\m}(\Mc)$ and $\AF$ be a pseudodifferential order $-\Nb/2$ operator on $\Mc$. Consider  operator $\Tb=\Tb_{P,\AF}$, $P=V\mu$. Then for the eigenvalues of this operator  estimate \eqref{BasicEstimate1} is valid, with constant not depending on the density $V.$
 \end{cor}
 In a quite standard proof,
  we consider a finite covering by  neighborhoods $\Uc_j$ with co-ordinate mappings to domains in the Euclidean space. The measure $P$ thus splits into the sum $P=\sum P_j$, where $P_j$ is supported inside $\Uc_j$. Operator $\AF^*P\AF$ thus splits into the sum $\sum \AF^*P_j\AF$, and the required eigenvalue estimate follows from estimates for these summands by means of  Ky Fan's inequality. In its turn, eigenvalue estimate for $\AF^*P_j\AF$
follows from the Euclidean result by usual localization.

\section{Examples, applications}\label{Sect.Examples.est}
\subsection{One-dimensional examples} The results about estimates are nontrivial even in dimension 1. As an illustration, we consider the weighted Steklov (Dirichlet-to Neumann) and transmission spectral problems with  weight being singular measure. Such problems, with  weight being a function in the Orlicz class, were considered in \cite{ShStokes}.

Let $\Om\subset\R^2$ be a bounded simply connected domain with \emph{smooth} boundary. We suppose that  $\m$ is a measure on the boundary $\Si=\partial\Om$ and  $V$ is a $\m$-measurable real function on $\Si$, $P=V\m$. We consider the eigenvalue problem
\begin{equation}\label{weighted.Steklov}
    \Delta u(X)=0, \, X\in\Om; u(X) P=\la \partial_\n(X) u(X), X\in\Si,
\end{equation}
where $\n(X)$ is the external normal at $X\in\Si.$ Equation \eqref{weighted.Steklov} understood in the sense of distributions.
This problem admits the following exact formulation. We denote by $\Dc\!\!\Nc$ the Dirichlet-to-Neumann operator on $\Si$, namely
\begin{equation*}
    (\Dc\!\!\Nc h)(X)= \partial_{\n(X)} u(X), \,\mbox{where} \, \Delta u=0, u|_{\Si}=h.
\end{equation*}
It is known that $\Dc\!\!\Nc$ is an order $1$ positive elliptic pseudodifferential operator on $\Si,$ with principal symbol $|\x|, $ $(X,\x)\in T^*\Si$.  Then the problem \eqref{weighted.Steklov} can be expressed as
\begin{equation*}
    \la \Dc\!\!\Nc h=P h,
\end{equation*}
or, in our variational setting, the eigenvalue problem for the operator $\Tb_{P,\AF}$ defined by the quadratic form
\begin{equation}\label{WeightedSt.2}
    \tb_{P,\AF}[h]=\int_{\Si}|(\AF h)(x)| P(dx), \, \AF=(\Dc\!\!\Nc)^{-\frac12}.
\end{equation}
Suppose that  measure $\m$ is $s-$Ahlfors regular of dimension $s\in(0,1]$ (the case of $s=1$ corresponds to the measure being absolutely continuous with respect to the Lebesgue measure on $\Si$.) Then Theorem \ref{EstTheor} gives us the following eigenvalue estimate.
\begin{cor}\label{CorStekl} Let $V$ be a real function in $L^{\Psi,\m}$. Then for the eigenvalues of the problem \eqref{weighted.Steklov},
\begin{equation*}
    n_{\pm}(\la)\le C \la^{-1}\|V_{\pm}\|^{\Psi,\m}_{\Si}.
\end{equation*}
\end{cor}

A similar result is valid for the transmission problem. Let, again, $\Om\subset\R^2$ be a bounded, simply connected domain with smooth boundary and $\Si$ be a  simple smooth curve inside $\Om$. For a function $u\in H^2(\Om\setminus \Si)\cap H^1(\Om),$
we denote by $[u_{\n(X)}]$ the jump of the normal derivative of $u$ at the point $X\in\Si.$ As above, $\m$ is a measure on $\Si$, $P=V\mu.$
We consider the spectral transmission problem

\begin{equation}\label{trans.weight}
  \D u=0,\, X\in\Om\setminus\Si;\,  u(X)P=\la [u_{\n(X)}], X\in \Si; \, u_{\partial \Om}=0\, \mbox{in} \, \Si.
\end{equation}
This kind of transmission problems is considered, e.g., in \cite{Agr}, \cite{AgrKatz}, motivated, in particular, by some physics applications. Similar to the reasoning above,
problem \eqref{trans.weight} can be transformed to the eigenvalue problem for the operator $\Tb_{P,\AF},$
defined in $L_2(\Si)$ by the quadratic form \eqref{WeightedSt.2},
where $\AF=\TF^{-\frac12}$ and $\TF$ is  the 'transmission operator'
 $$\TF: h\mapsto [u_{\n(X)}], X\in \Si;\, \D u= 0 \,\mbox{in} \, \Om\setminus\Si, u|_{\partial\Om}=0; u|_{\Si}=h.$$
 Again, $\AF$ is an order $-\frac12$ pseudodifferential operator on $\Si$, and the spectral problem fits in our general setting.

\begin{cor}\label{CorSteklAs} Suppose that  measure $\m$ is Ahlfors regular of dimension $s\in(0,1)$ and let $V$ be a real function in $L^{\Psi,\m}$. Then for the eigenvalues of the problem \eqref{trans.weight},
\begin{equation}\label{EstimSteklov}
    n_{\pm}(\la, \Tb_{P,\AF})\le C \la^{-1}\|V_{\pm}\|^{\Psi,\m}_{\Si}.
\end{equation}
\end{cor}

In the case $s=1$ for the weighted Steklov problem, this kind of estimates was obtained in \cite{ShStokes}. There, for $V\ge 0,$ a lower estimate for  $n_{+}(\la, \Tb_{P,\AF})$
was established as well, of the same order in $\la$ but in terms of the $L_1$ norm of the function $V$. General lower estimates for eigenvalues are discussed later on, in Section \ref{lower}.

 \subsection{Fractal sets} We recall the general construction of fractal sets, introduced  by J.Hutchinson, \cite{Hu}. Let $ \pmb{\Sc} = \{\Sc_1, ...\Sc_m\}$ be a finite collection  of contractive similitudes (i.e., compositions of a parallel shift, a linear isometry and a contracting homothety)
on $\R^{\Nb}$, $h_1, ..., h_m$ are their coefficients of contraction.  It is supposed that \emph{the open set condition} is satisfied: there exists an open set $\Vs\subset \R^\Nb$
 such that $\cup \Sc_\io(\Vs)\subset \Vs$ and $\Sc_\io(\Vs)\cap \Sc_{\io'}(\Vs)=\varnothing, \io\ne \io'.$ By the results of Sect. 3.1 (3), 3.2 in \cite{Hu}, there exists a unique compact set $\K=\K(\pmb{\Sc})$ satisfying $\K=\cup_{j\le m} \Sc_j\K.$ This set is, in fact,  the closure of the set of all fixed points of finite compositions of the mappings $\Sc_j.$ The Hausdorff dimension $d$ of the set $\K(\pmb{\Sc})$ is determined by the equation $\sum h_j^d=1 .$ Let $\m$ be the $d$-dimensional Hausdorff measure $\m_{\pmb{\Sc}}$ on $\K(\pmb{\Sc})$. As explained in \cite{Fra}, Corollary 2.11.(1), p.6696, this measure is Ahlfors regular of dimension $d$. Therefore, our result, Theorem \ref{EstTheor}, gives the upper eigenvalues estimate:
\begin{cor}\label{FractESt} Let $\m=\m( \pmb{\Sc})$ be a fractal measure as above, with bounded set $\Om\subset\R^{\Nb}$. Suppose that a density $V$ belongs to the Orlicz space $L^{\Psi,\mu}$; $P=V\mu$ and $\AF$ be an order $-l=-\Nb/2$ pseudodifferential operator in $\Om$ with compact support. Then  operator $\Tb_{P,\AF}$ belongs to $\SF_{1,\infty}$  and for   its eigenvalues  the following estimate holds
\begin{equation*}
   n_{\pm}(\la, \Tb_{P,\AF})\le C \la^{-1} \|V_{\pm}\|_{\K}^{(av,\Psi,\m)}.
\end{equation*}
\end{cor}
\subsection{Lipschitz surfaces} Let the set $\Si\subset \R^\Nb$ be a compact Lipschitz surface of dimension $d>0$ and codimension $\dF=\Nb- d.$ Recall that this means that locally $\Si$ can be, in proper co-ordinates $X=(\xb;\yb)=(x_1,\dots,x_d;y_1,\dots y_\dF),$ represented as $\yb=\pmb{\phi}(\xb), \, \xb\in G \subset\R^{d}$ with a Lipschitz $\dF$-component vector-function $\pmb{\phi}.$  Denote by $\m_\Si$ the measure on $\Si$ generated by the embedding $\Si\to\R^\Nb$ - it coincides with the $d$ - dimensional Hausdorff measure $\Hc^{d}$. By the Rademacher theorem the gradient of $\pmb{\phi}$ exists almost everywhere with respect to $\m_\Si$. So, locally, $\m_\Si$ has the form
 $$d\m_\Si=\s(\xb)d\xb, \, \s(\xb)=\det(\pmb{1}+(\nabla\pmb{\phi})^*\nabla\pmb{\phi})^{\frac12}.$$

  The embedding of $\Si$ into $\R^\Nb$ generates a singular measure on $\R^\Nb$, supported on $\Si$ which will be also denoted
$\m_\Si$, as long as this does not cause confusion. Such measures satisfy condition \eqref{Aregular} with $s=d$. Therefore, for measure $P=V\!\mu_{\Si},$ the eigenvalue estimates obtained in Sect. \ref{Sect.Est} hold. For further reference, we formulate two important cases.
\begin{thm}\label{Est.Lip.Th} Let $\Si$ be a $d$-dimensional compact Lipschitz surface in $\R^\Nb$ and $V\in L^{\Psi,\m}(\Si).$
If $\AF$ is an order $-l=-\Nb/2$ pseudodifferential operator in $\R^{\Nb}$ with compact support or $\AF=(1-\D)^{-l/2}$, then for the eigenvalues of the operator $\Tb_{V,\m_\Si,\AF}$ the eigenvalue estimate is valid:
\begin{equation}\label{est.Lip.}n_{\pm}(\la,\Tb_{V,\m_\Si,\AF})\le C(\m)C(\AF)\la^{-1}\|V\|^{(av,\Psi,\m)}.
\end{equation}
\end{thm}
Localization of the first case of Theorem \ref{Est.Lip.Th} provides us with an eigenvalue estimate for operator on compact manifolds.
\begin{cor}Let $\Mc$ be a smooth closed Riemannian manifold of dimension $\Nb$ and $\Si$ be a $d$-dimensional compact Lipschitz surface in $\Mc$. Then for $V\in L^{\Psi,\m}(\Si)$ and $\AF=(1-\D_{\Mc})^{-\Nb/4},$  estimate \eqref{est.Lip.} holds.
\end{cor}
More about operators on manifolds can be found in Section \ref{Sect.surface}.
\subsection{Logarithmic potential} Here we demonstrate the relation of our construction with the logarithmic potential operator. A logarithmic potential of a measure $P$ in $\R^\Nb$ is usually defined as
\begin{equation*}
    \Lb[P](X)=\int \log|X-Y| P(dY).
\end{equation*}
This  object is being extensively used in Potential Theory, Analysis, and Partial Differential Equations, as well as numerous applications. We take a somewhat different point of view on this potential. Let a compactly supported finite Borel measure $\m$ be $s$-Ahlfors regular, $s>0$,  and with $P=V\m$ and $V\in L^{\Psi,\m}$, $V\ge0,$ we associate the logarithmic potential as an operator in the space $L_{2,P}(\R^\Nb)$:
\begin{equation}\label{LogOperator}
    \LF_P: L_{2,P}\to L_{2,P};\, \LF_{P}: f(X)\mapsto  \int \log|X-Y|f(Y) P(dY), f\in L_{2,P}.
\end{equation}
\begin{thm}\label{LogTheorem} Operator $\LF_P$ is a bounded self-adjoint operator in $L_{2,P}$; it is compact and for the distribution function $n(\la,\LF_{P})$ of  its singular numbers $s_k(\LF_P) $ the estimate holds
\begin{equation}\label{EstLog}
    \limsup_{\la\to 0} \la n(\la,\LF_{P})\le C(\m)\|V\|^{(av,\Psi,\mu)}
\end{equation}
\end{thm}
\begin{proof}We apply the transformation used already once in Sect. \ref{Bddness.Sect} (and to be used again in the study of eigenvalue asymptotics.) Consider  operator $\Tb_{V,\m,\AF}$ under the conditions of Theorem \ref{LogTheorem}, with a special choice of  $\AF:\,$ namely  $\AF=(1-\Delta)^{-\Nb/4}$.  Similar to Section \ref{Bddness.Sect},  operator $\Tb_{V, \m,\AF}$ admits representation
\begin{equation*}
\Tb_{V,\m, \AF} = \KF^*\KF,
\end{equation*}
with  $\KF$  acting from $L^2(\R^\Nb)$ to $L^2(\Mb,\mu), \, \Mb=\supp\m, $ as
$ \KF= U\G_\Mb\AF,$ where $\G_\Mb$ is the restriction from $\R^\Nb$ to $\Mb$, a bounded operator from $H^{\Nb/2}(\R^\Nb)$ to $L_{2,\m}$,  $U=V^{\frac12}$ and the composition is bounded. Moreover, under our conditions, by Theorem \ref{EstTheor},
\begin{equation}\label{LogEstimatePrel}
n(\la,\KF^*\KF)\le \la^{-1} C(\m) \|V\|^{(av,\Psi,\m)}(\Mb).
\end{equation}

Operator $\KF^* =\AF^*\G_\Mb^* U: L^2(\Mb,\m)\to L^2(\R^\Nb)$ should be understood as composition of $\G_\Mb^* U$ acting, after the multiplication by $U$, as the extension by zero outside $\Mb$ to the space of distributions $H^{-l}(\R^\Nb)$, and  the pseudodifferential order $-l$ operator $\AF^*$ which maps $H^{-l}(\Om)$ to $L^2(\R^\Nb)$.

Now, recall that  nonzero eigenvalues of non-negative operators $\KF^*\KF$ in $L^2(\R^{\Nb})$ and $\KF\KF^*$ in $L^2(\Mb,\m)$ coincide. The operator $\KF\KF^*$ acts as
\begin{equation}\label{KKstar}
    \KF\KF^* =  U\G_\Mb\AF  \AF^*\G_\Mb^* U =U\G_\Mb(\AF  \AF^*)\G_\Mb^* U.
\end{equation}
Here operator $\AF  \AF^*$ is an order $-2l=-\Nb$ pseudodifferential operator which we consider as acting from $H^{-l}(\R^\Nb)$ to $H^l(\R^{\Nb})$. It has principal symbol $|\X|^{-\Nb}$, and therefore, it is the integral operator with logarithmic principal singularity of the kernel $R(X,Y,X-Y)$:
\begin{equation}\label{KKstarKernel}
 R(X,Y,X-Y) = \Cb_\Nb\log|X-Y|+ R'(X,Y)
\end{equation}
with $R'(X,Y)=o(1), X\to Y$. The coefficient $\Cb_\Nb$ equals $\frac{2\sqrt{\pi}}{\G(\Nb/2+1)}$ (see, e.g. \cite{Schwartz}, (VII.7.15)).
Therefore,  operator $\KF\KF^*$ acts, up to weaker terms, as
\begin{equation}\label{ActionKKstar}
    (\KF\KF^* v)(X) =\Cb_\Nb U(X)\int_{\Mb}\log|X-Y|U(Y) v(Y) d\m(Y).
\end{equation}
in $L^2(\Mb,\m)$. Finally, the eigenvalue problem $\KF\KF^* v=\la v$ in $L^2(\MF,\mu)$, by setting $v(X)=U(X)f(X)$, transforms to the eigenvalue problem    \eqref{LogOperator}
for operator of logarithmic potential. Eigenvalue estimate  \eqref{EstLog} follows therefore from  \eqref{LogEstimatePrel}.
\end{proof}
In the next section we benefit of the above way of reasoning  acting in the opposite direction.
\section{Eigenvalue asymptotics and measurability. Lipschitz surfaces}\label{Sect.As}
The measurability of the Birman-Schwinger type operator $\Tb_{V,\m}$ can be derived, in particular, from the eigenvalue asymptotics for this operator. Note that results stating such  asymptotics are much stronger than just  measurability. Nevertheless, in all approaches to proving measurability of this type of operators, the eigenvalue  asymptotics itself, or at least some weaker version  of it, like the Wodzicki residue, serve as  the starting point. It seems that for a long time, specialists in Noncommutative Geometry, when dealing with Connes' measurability, were unwary of  publications by M.Sh. Birman and M.Z. Solomyak in late 70-s on the eigenvalue asymptotics for negative order pseudodifferential operators as well as of further extensions in this direction. It turns out that these  results and their consequences, in particular, for potential type integral operators, enable one to establish integrability in a considerably more general setting.

In this section and the next one, we will systematically use a perturbation technique based upon the fundamental asymptotic perturbation  lemma by M.Sh.Birman and M.Z.Solomyak, see, e.g., Lemma 1.5 in \cite{BS}. For Readers' convenience, we  reproduce the formulation (by far, not the most general one) we need further on.

\begin{lem}\label{BSLemma} Let $\Tb$ be a self-adjoint compact operator. Suppose that for $\ve$ small enough, $\Tb$ can be represented as a sum, $\Tb=\Tb_\ve+\Tb_\ve', $ so that for eigenvalues of $\Tb_\e$ the asymptotics is known, $n_{\pm}(\la,\Tb_{\ve})\sim \la^{-1}C^{\pm}_{\ve}, \la\to 0,$ while for the singular values of $\Tb_\ve',$ the asymptotic estimate is valid, $\limsup_{\la\to 0}\la n(\la,\Tb)\le \ve$. Then the limits $C^{\pm}=\lim_{\ve\to 0}C^{\pm}_{\ve}$ exist and for the eigenvalues of $\Tb$ the asymptotic formulas hold, $n_{\pm}(\la,\Tb)\sim \la^{-1}C^{\pm}, \,\la\to 0.$
\end{lem}
\subsection{Measures on Lipschitz surfaces}\label{Subsect.Lip}
Formulas for the eigenvalue asymptotics for a measure on a Lipschitz surface were obtained in \cite{RSh2}. We discuss them here briefly and then present certain generalizations.

Let $\Si\subset\R^\Nb$ be a compact Lipschitz surface of dimension $d: 0<d<\Nb, \dF=\Nb-d,$ defined locally, in proper co-ordinates $X=(\xb,\yb), \xb\in \R^d, \yb\in \R^{\dF}$  by the equation $\yb=\pmb{\f}(\xb)$, with a Lipschitz vector-function $\pmb{\f}$. Measure $\m$ generated by the embedding of $\Si$ into $\R^{\Nb}$ coincides with the $d$-dimensional Hausdorff measure $\Hc^d$.  By the Rademacher theorem, for $\m$-almost every  $X\in\Si$, there exists a tangent space $T_X\Si$ to $\Si$ at the point $X$ and, correspondingly, the normal space $N_X\Si$ which are identified naturally with the cotangent and the conormal spaces. By $S_X\Si$ we denote the sphere $|\x|=1$ in $T_X\Si.$
 \begin{thm}\label{Th.As.RSh} Let the real  function  $V$ on $\Si$ belong to $L^{\Psi,\m}(\Si)$. Let $\AF$ be a compactly supported in $\Om\subset\R^\Nb$ order $-l=-\Nb/2$ pseudodifferential operator with principal symbol $a_{-l}(X,\X)$.
At the points $\X\in \Si$  where the tangent plane exists we define the auxiliary symbol $r_{-d}(X,\x)$, $\x\in T_X\Si$,
\begin{equation*}
    r_{-d}(X,\x)=(2\pi)^{-\dF}\int_{N_X\Si}|a_{-l}(X,\x,\y)|^2 d\y.
\end{equation*}
and the density
\begin{equation}\label{ro} \ro_{\AF}(X)=\int_{S_X\Si} r_{-d}(X,\x)d\x
\end{equation}
Then for the eigenvalues of  operator $\Tb_{V,\m,\AF}=\AF^{*}P\AF,$ $P= V\m$, the asymptotic formulas are valid
\begin{equation}\label{AsLambda}
     n_{\pm}(\la,\Tb_{V,\m,\AF})\sim \la^{-1} A_{\pm}(V,\m,\AF), \, \la\to 0,
    \end{equation}
where
  \begin{gather*}
   A_{\pm}(V,\m,\AF)= \frac{1}{d (2\pi)^{d-1}}\int_\Si\int_{S_X\Si} V_{\pm}(X)r_{-d}(X,\x)d\x d\m= \\\nonumber
   \frac{1}{d (2\pi)^{d-1}} \int_\Si V_{\pm}(X)\ro_{\AF}(X) d\m(X), \end{gather*}
with  density $\ro_{\AF}(X)$ defined in \eqref{ro}.
\end{thm}
The expression
 \begin{equation}\label{Resid}
 A(V,\m,\AF)= A_{+}(V,\m,\AF)-A_{-}(V,\m,\AF)=\frac{1}{d (2\pi)^{d-1}}\int_{\Si}V(X)\ro_{\AF}(X)d\m(X)
 \end{equation}
can be formally understood as an analogy of the  Wodzicki residue of the symbol $V(X)r_{-d}(X,\x)$ on $\Si$, of course, without any smoothness conditions inherent to Wodzicki theory (the latter 'symbol' is even not expected to be a symbol of any pseudodifferential operator).  We call it $\Si$-\emph{Wodzicki residue} of $(V,\AF,\m).$

In the particular case of  $\AF=\AF_0=(1-\D)^{-\Nb/4},$ in a neighborhood of $\Si$ in $\R^\Nb$, we have  $a_{-l}(X,\X)=|\X|^{-\Nb/2}$ and

\begin{gather*}
r_{-d}(X,\x)= (2\pi)^{-\dF}\int_{\R^{\dF}}(|\x|^2+|\y|^2)^{-\Nb/2}d\y =\\ \nonumber |\x|^{-d}(2\pi)^{-\dF}\pmb{\om}_{\dF-1}\int_{0}^{\infty}\z^{\dF-1}(1+\z^2)^{{-\Nb/2}} =
\\ \nonumber \pmb{\om}_{\dF-1}\frac{1}{2(2\pi)^\dF}\Bb\left( \frac{\dF}{2}, \frac{d}{2}\right) |\x|^{-d},
\end{gather*}
where $\pmb{\om}_{\dF-1}$ is the volume of the unit sphere in $\R^{\dF}$, $\Bb$ is the Euler  Beta-function.
So, here we have
\begin{equation}\label{integral}
   n_{\pm}(\la, \Tb_{V,\AF_0})\sim \la^{-1}\Zb(d,\dF)\int_{\Si} V_{\pm}(X) d\m(X),
\end{equation}
\begin{equation}\label{IntCoeff}
\Zb(d,\dF)=\frac{\pmb{\om}_{\dF-1}\pmb{\om}_{d-1}}{2d(2\pi)^{\dF}}\Bb\left( \frac{d}{2}, \frac{\dF}{2}\right)
\end{equation}

We explain briefly the way how Theorem \ref{Th.As.RSh} is being proved, directing  interested Readers to \cite{RSh2} for details.

First, we can  replace $V$ by a  weight $V_{\ve}$, defined and smooth in a neighborhood of $V,$ such that the eigenvalue distribution functions for operators $\Tb_{V,\m,\AF}$ and $\Tb_{V_{\ve},\m,\AF}$ differ asymptotically  by less than $\ve\la^{-1}$. Here estimates in Section \ref{Sect.Est} are used. By the basic asymptotic perturbation lemma by M.Sh. Birman and M.Z. Solomyak, see Lemma \ref{BSLemma}, such approximation enables one to prove asymptotic formulas for nice densities $V_{\ve}$ only, by passing then to limit as $V_{\ve}$ approaches $V$ in the averaged Orlicz norm. On the next step, we separate the positive and negative eigenvalues of our operator. Namely, by some more approximation and localization, we find that, in the leading term, the asymptotics of   positive eigenvalues of the operator is determined only by the positive part of the density $V_{\ve}$, while the asymptotics of the negative eigenvalues is determined only  by the negative part of $V_\ve$. Thus, the problem is reduced to
the case of a sign-definite  $V_{\ve}$, which we may suppose being the restriction to $\Si$ of a smooth sign-definite function.

  Next, the problem is reduced to the study of eigenvalues of an integral operator on $\Si$ with kernel having an order zero and/or logarithmic singularity at the diagonal. This is done in the following way. Similarly to \eqref{factorization}, operator $\Tb_{V_{\ve},\m,\AF}$ factorizes as
\begin{equation}\label{factorizationSi}
 \Tb_{V_{\ve},\m,\AF}=(\G_{\Si}U\AF)^*(\G_{\Si}U\AF)
 \end{equation}
where $U=V_{\ve}^{\frac12}$, $\G_{\Si}$ is the operator of restriction from $\Om$ to $\Si$, so the operator $\KF=\G_{\Si}U\AF$ is bounded as acting from $L^2(\Om)$ to $L^2(\Si,\m)$ and the product $\KF^*\KF=(\G_{\Si}U\AF)^*(\G_{\Si}U\AF)$ acts in $L^2(\Om)$. We know, however, that the nonzero eigenvalues of the operator $\KF^*\KF$ coincide with nonzero eigenvalues of $\KF \KF^*,$ counting multiplicities. Operator $\KF \KF^*$ acts in $L^2(\Si,\m)$ as

\begin{equation}\label{transposed}
\KF \KF^*=   \G_{\Si}U\AF\AF^*U\G_{\Si}^*.
\end{equation}
Operator $U\AF\AF^*U$ is an order $-\Nb$ pseudodifferential operator in $\Om$ with principal symbol $\Rc_{-\Nb}(X,\X)=V_{\ve}(X)|a_{-l}(X,\X)|^2$, or, equivalently, as a self-adjoint integral operator with kernel $R(X,Y,X-Y)$, smooth for $X\ne Y$. This kernel, being the Fourier transform of the symbol of $U\AF\AF^*U$ in $\X$ variable, has the leading singularity in $X-Y$ containing possible terms of two types, namely, $R_0(X,Y,X-Y),$ order zero homogeneous in $X-Y$, and $R_{\log}(X,Y)\log|X-Y|$ with smooth function $R_{\log}$ - see, e.g., \cite{Taylor 2}, Ch. 2, especially, Proposition 2.6. Note that one of these terms may be absent. In particular, if the principal symbol of $\AF$ equals $|\X|^{-\Nb/2},$ this means that  $\AF$ is $(1-\Delta)^{-\Nb/4},$ framed, possibly, by cut-off functions -- and this is the most interesting case--, only the logarithmic term is present.
After framing by $\G_{\Si}U$ and $U\G_{\Si}^*$, as in \eqref{transposed}, we arrive at the representation of $\KF\KF^*$ as the integral operator  $\RF$ in $ L^2(\Si,\m)$ with kernel $R(X,Y,X-Y)=R_0(X,Y,X-Y)+R_{\log}(X,Y)\log|X-Y|.$ Exactly this kind of operators on Lipschitz surfaces was considered in the papers \cite{RT1}, for  surfaces of codimension 1, and \cite{RT2}, for an arbitrary codimension. The result on eigenvalue asymptotics, obtained for such integral operators in these papers, corresponds exactly the formulas in Theorem \ref{Th.As.RSh} above.

For an interested Reader we explain now, not going into technical  details (which are presented in \cite{RT1}, \cite{RT2}), how the formulas for eigenvalue asymptotics for integral operators on Lipschitz surfaces are  being proved. The starting point is establishing these formulas for a smooth surface. This is achieved by an adaptation of  the results by M.Sh.Birman and M.Z.Solomyak in \cite{BirSolPDO1} on the eigenvalue asymptotics for negative order pseudodifferential operators. Next, the given Lipschitz surface  $\Si:\yb=\pmb{\f}(\xb)$ is approximated, locally,  by smooth ones, $\Si_\e,$ so that in their local representation $\yb=\pmb{\f}_\e(\xb)$, functions $\pmb{\f}_\e$ converge to $\pmb{\f}$ in $L^{\infty}$ and their gradients $\nabla\pmb{\f}_\e$ converge to $\nabla\pmb{\f}$ in all $L^p, p<\infty$ (one should not expect convergence of gradients in $L^\infty,$ of course). The changes of variables $\xb\mapsto (\xb,\pmb{\f}(\xb))$, resp., $\xb\mapsto (\xb,\pmb{\f}_{\e}(\xb))$, transform operators with kernel $R(X,Y,X-Y)$ on the surfaces $\Si$ and $ \Si_\e$ to operators $\RF,$ resp., $\RF_{\e},$ on some domain in $\R^{d}$, while the eigenvalue asymptotics for  ${\RF}_{\e},$ is known. Now it is possible to consider the difference of these operators. After estimating the eigenvalues of ${\RF}-{\RF}_{\e}$ using the closeness or $\pmb{\f}$ and $\pmb{\f}\e$ (and this is a fairly technical part of the reasoning), we obtain that the eigenvalue asymptotic coefficients of $ {\RF} -{\RF}_{\e}$ converge to zero. This property enables one to use again the asymptotic perturbation lemma, Lemma \ref{BSLemma} to justify the eigenvalue asymptotics formula for ${\RF}.$

\subsection{Connes' integral over a Lipschitz surface}\label{Sub.Connes} As soon as Theorem \ref{Th.As.RSh} is proved, it follows immediately that the operator $\Tb_{V,\m,\AF}$ is Connes' measurable.

 \begin{thm}\label{measureoneLip} Let $\Si$ be a compact $d$--dimensional Lipschitz surface in $\R^{\Nb},$ with the induced  measure $\m=\Hc^d,$ and $\AF$ be a compactly supported  order $-\Nb/4$ pseudodifferential operator in $\R^\Nb$ or $\AF=(1-\Delta)^{-\Nb/4}$. Then for any $V\in L^{\Psi,\m}$, operator $\Tb_{V, \m,\AF}$ is Connes' measurable and $\ta(\Tb_{V, \m,\AF})$ equals the $\Si$-Wodzicki residue of $(V,\Si,\m,\AF)$, for any normalized positive singular trace $\ta$ on $\MF_{1,\infty},$
 \begin{equation*}
   \ta(\Tb_{V,\m,\AF})=d^{-1}(2\pi)^{-d}\int_{\Si}V(X)\ro_{\AF}(X)\mu(dX).
 \end{equation*}

 \end{thm}
 \begin{proof}
 In fact, since the weak Schatten ideal $\SF_{1,\infty}$ is embedded in $\MF_{1,\infty}$ then,  for a sign-definite parts of the density,  $V_+$ or $V_-$, the asymptotic relations
\begin{equation*}
   (\log(2+n))^{-1}\sum_{k\le n}\la_{k}(\Tb_{V_{\pm},\m,\AF})\to A_{\pm}(V_{\pm},\m,\AF),\, n\to\infty
\end{equation*}
are valid, being a direct consequence of  \eqref{AsLambda}. Therefore
\begin{equation*}
\ta(\Tb_{V_{\pm},\m,\AF})=A_{\pm}(V_{\pm},\m,\AF)
\end{equation*}
for any normalized positive singular trace $\ta$ on $\MF_{1,\infty}$ for any $V_+$, resp., $V_-$ in the Orlicz  space $L^{\Psi,\m}$.  For $V$ with variable sign, we can use our Theorem \ref{Th.As.RSh} in its whole strength. Having the asymptotics \eqref{AsLambda}, both for positive and negative eigenvalues of $\Tb_{V,\m,\AF}$, we find that
\begin{gather}\label{Trace.pm}
   \ta(\Tb_{V,\m,\AF})=\ta(\Tb_{V_+,\m,\AF})-\ta(\Tb_{V_-,\m,\AF})= \lim(\log(2+n))^{-1}\sum_{|\la_k^{\pm}|<n}\la_k^{\pm} = A(V,\m,\AF)\\\nonumber=A^+(V,\m,\AF)-A^-(V,\m,\AF)=\frac{1}{d (2\pi)^{d-1}}\int_\Si\int_{S_X\Si} V(X)r_{-d}(X,\x)d\x d\m ,
\end{gather}
for any normalized singular trace on $\MF_{1,\infty}.$
This, according to definition, means that the operator $\Tb_{V,\m,\AF}$ is Dixmier-Connes' measurable. In particular, if we select $\AF=(1-\D)^{-\Nb/4},$ Connes' integral of the operator  $\Tb_{V,\m,\AF}$ coincides, up to a constant factor  in \eqref{IntCoeff} depending on the dimensions $\dF$ and $d$ only, with the surface integral of $V$ against the measure $\m$ on the Lipschitz manifold $\Si,$ see Section \ref{Subsect.Lip}.
\end{proof}

\subsection{Finite unions of Lipschitz surfaces in $\R^{\Nb}$. Measurability}\label{subs.samedim.trace}
Now we discuss Connes integrals over sets of more complicated structure. Here, the general result as in Theorem \ref{measureoneLip}, although possible,  is not that visual due to the dependence of the local formula in the integrand on a particular  representation of  Lipschitz surfaces involved. Therefore, from now on, we restrict ourselves to the analysis of the special case of operator $\AF=(1-\Delta)^{-\Nb/4}$ in $\R^\Nb$ having principal symbol $a_{-l}(X,\X)=|\X|^{-l}.$ We will omit $\AF$ in notation further on.

Our aim is to arrive  at the formula $\ta(\Tb_{P})=C\int P$ for widest  possible set of measures.

Let $\XF\subset \R^{\Nb}$ be a compact set, $\XF=\cup_{j\le \nb}\Si_j$ where each $\Si_j$ is a compact Lipschitz surface of dimension $d,\, 0<d<\Nb$. With $\Si_j$ we associate the measure $\m_j$ supported on $\Si_j$  generated by the embedding $\Si_j\subset \R^{\Nb}.$   We normalize these measures, setting $\tilde{\m}_j=\Zb(d,\dF)\m_j$, the coefficient $\Zb(d,\dF)$ given in \eqref{IntCoeff}. Let further $V_j$ be  real-valued functions on $\Si_j$, belonging to the corresponding Orlicz spaces, $V_j\in L^{\Psi, \m_j},\, P_j=V_j\m_j$. In our normalization, we associate measure $\tilde{P}=\sum_j V_j\tilde{\m}_j$ with the given measure $P=\sum_j V_j\m_j$.  This relation will be denoted by $\NF: P\mapsto \tilde{P}.$ This operator $\NF$ is extended by linearity to sums of measures supported on surfaces of different dimension.

With each of surfaces we associate operator $\Tb_{P_j}.$ In accordance with \eqref{Trace.pm},
\begin{equation*}
    \ta(\Tb_{P_j})= \int_{\Si_j}V_j(X)\tilde{\m}_j(dX)=\int \tilde{P}_j(dX)=\int \NF(P_j)(dX).
\end{equation*}
Thus, with our normalization,  we have a convenient expression for the Connes integral over the union of surfaces.

\begin{thm}\label{FinitProp} Let measure $P$ be defined as $P=\sum V_j\m_j$ with $\m_j$ being the $d$-dimensional Hausdorff measure on a compact Lipschitz surface $\Si_j,$ $V_j\in L^{\Psi,\m}.$  Under the above conditions, operator $\Tb_{P}$ satisfies
\begin{equation}\label{Finite sum}
    \Tb_{P} =\sum \Tb_{P_j},
\end{equation}
it is Dixmier-Connes measurable, and for any normalized singular trace $\ta$,
  \begin{equation}\label{FiniteLinear}
    \ta(\Tb_{P})=\sum_j A(P_j)=\int \NF(P)(dX).
  \end{equation}
\end{thm}
\begin{proof} Relation \eqref{Finite sum} follows from the corresponding formula for the quadratic forms of the operators involved. The linearity property of singular traces implies \eqref{FiniteLinear}, and, since the expression on the right does not depend on $\ta$, measurability of the operator follows.
\end{proof}

\subsection{Finite unions of Lipschitz surfaces in $\R^{\Nb}$ of the same dimension. Eigenvalue asymptotics}\label{Subs.samedim.as}
The statement in Theorem \ref{FinitProp} is considerably weaker than  the one concerning the eigenvalue asymptotics,    namely, that
 \begin{equation}\label{addit.asymp}
    n_{\pm}(\la, \Tb)\sim \sum n_{\pm}(\la, \Tb_{P_j})\sim \la^{-1}\Zb(d,\dF)\sum \int V_{j,\pm}(X)  dm_j(X)
 \end{equation}
 holds. This is understandable, since, unlike the singular trace, the coefficients in the eigenvalue asymptotics do not, generally, depend linearly on the operators. Moreover, simple examples show that \eqref{addit.asymp} may be wrong, unless we impose some additional conditions. In particulaar, it was established in \cite{RSh2}, Theorem 7.1, that \eqref{addit.asymp} is correct provided we suppose that a rather restrictive additional condition is satisfied,
 namely, that  surfaces $\Si_j$  are disjoint.

 However,  properly modified, \eqref{addit.asymp} is still correct. In order to formulate it, we introduce, for  given Lipschitz surfaces $\Si_j,$ $j=1,\dots,\nb$ in $\R^\Nb$ and real densities $V_j\in L^{\Psi,\m_j}(\Si_j),$ the signed measure
  \begin{equation}\label{sum.mesure}
  P=\sum P_j=\sum V_j \m_j
 \end{equation}
(measures $\m_j$, $P_j$ and densities $V_j$ are extended, as usual,  to $\R^{\Nb}$ by zero,
 \begin{equation*}
  P_j(E):=P_j(E\cap \Si_j)=\int_{E\cap\Si_j}V_j(X)d\m_j(X),
 \end{equation*}
 for a Borel set $E\subset \R^{\Nb}.$)
 A visual picture of $P$ is the following. Let $\m$ be the $d$-dimensional Hausdorff measure on $\XF=\cup \Si_j$. For each point $X\in\XF$, we define $V$ as $\sum V_j(X),$ over thise $j$ for which $X\in\Si_j.$ The standardly defined  positive and negative parts of measure $P$ equal $P_{\pm}=V_{\pm}(X)\m.$
 \begin{thm}\label{Th.Add.Asymp} In the above notations
 \begin{equation}\label{SumMeasures}
    \lim_{\la\to 0} \la n_{\pm}(\la, \Tb_{P}) =\Zb(d,\Nb-d) \int_{\XF}P_{\pm}(dX)=\int_{\XF}\NF(P_{\pm})(dX).
 \end{equation}
 \end{thm}
  Of course, Theorem \ref{Th.Add.Asymp} is a considerably stronger assertion than the measurability theorem \ref{FinitProp}. Therefore it is not surprising that its proof is somewhat more technical. Readers interested only in Connes' integrability may skip the proof to follow.

  In proving Theorem \ref{Th.Add.Asymp}, we will use an important localization property established in \cite{RSh2}, see Lemma 3.1 there. Namely, if a measure $P$ is supported on two separated sets, i.e.,  $P=P^1+P^2$, $P^\io$ is supported in $\XF^\io$, $\io=1,2$, and the distance between the sets $\XF^1,\XF^2$ is positive, then
 \begin{equation}\label{separated supp}
    n_{\pm}(\la, \Tb_{P^1+P^2})-n_{\pm}(\la, \Tb_{P^1})-n_{\pm}(\la, \Tb_{P^2})=o(\la^{-1}),
 \end{equation}
 as $\la\to 0$. This can be understood as that in the case of separated measures,  up to a lower order error, the eigenvalues of $\Tb_{P^1+P^2}$ behave asymptotically as the eigenvalues of the direct sum of operators $\Tb_{P^\io}.$

 \begin{proof} In the proof we act by induction on the quantity $\nb$ of surfaces involved. As the base of induction,  for just one surface, the statement is contained in Theorem \ref{Th.As.RSh}.
 Now we explain informally the inductive step first, the details to follow further on.  Supposing that our statement is proved for the union $\XF$ of $\nb-1$ surfaces, we add one more surface, $\Si_{\nb}$ with density $V_\nb$. For a small $\ve>0$, we consider the $\ve$-neighborhood $\Gc_{\ve}$ of $\XF.$ The surface $\Si_{\nb}$  is split into three parts: the first one is the part of $\Si_{\nb}$ intersecting with $\XF$, the  second one is the part of $\Si_{\nb}$ lying in $\Gc_{\ve}$ but not in $\XF_{\nb}$, and the rest, the part of $\Si_{\nb}$  lying outside the neighborhood $\Gc_{\ve}$. Correspondingly, the measure $P_{\nb}$ splits into three parts. The first of these three measures, supported in $\XF,$ is added to the already given measure on $\XF$, and to the corresponding operator the inductive assumption is applied. This operator and the third one in the splitting of $P_{\nb}$ correspond to separated sets  the localization lemma 3.1 in \cite{RSh2} applies. As for the remaining measure on  the second one in the splitting, the eigenvalue counting function for the corresponding operator is small.

 Now, more formally, suppose that \eqref{SumMeasures} holds for $\nb-1$ surfaces $\Si_j, \, j=1,\dots,\nb-1,$ and we add one more surface, $\Si_{\nb}$ with density $V_\nb(X)\in L^{\Psi,\m_{\nb}}$, $P_{\nb}=V_{\nb}\m_{\nb}$, $P=\sum_{j\le\nb} P_j$. It is important to note that the set $\XF_{\nb}=\cup_{j\le\nb} \Si_j$ is Ahlfors $d$-regular.  If surface $\Si_{\nb}$ is disjoint with $\XF_{\nb-1}=\cup_{j<\nb}\Si_j$, these sets are separated (due to compactness) and our statement follows from \cite{RSh2},  Lemma 3.1, immediately. Now let   $\Si_{\nb}$ have a nonempty intersection $\Si'_{\nb}$ with $\XF_{\nb-1}$. Denote by $P_{\nb}'$ the restriction of the measure $P_{\nb}=V_{\nb}\m_{\nb}$ to $\Si'_{\nb}$ and by $\hat{P}_{\nb}$ the remaining part of $P_{\nb},$ i.e., the restriction of $P_{\nb}$ to the set $\Si_\nb\setminus  \XF_{\nb-1}$. Now we re-arrange our measures in the following way. We denote by $\check{P}$ the measure $\sum_{j<\nb}P_j+P_{\nb}',$  so
 \begin{equation*}
    P=\sum_{j\le\nb}{P_j}=\check{P} +\hat{P}_{\nb}.
 \end{equation*}
 Consider an $\ve$-neighborhood $\Gc_\ve$ of the set $\XF_{\nb-1}$. As $\e\to 0$,  Hausdorff measure $\Hc^d$ of the set $\Yc_\ve=(\Si_{\nb}\setminus\XF_{\nb-1})\cap \Gc_\ve$ tends to zero, therefore the averaged norm  $\|V_\nb\|^{\Psi,\m}_{\YF_\ve}$ tends to zero as well. We denote by $P_{\nb,\ve}$ the restriction of $P_{\nb}$ to the set $\YF_\ve$ and  by $P'_{\nb,\ve}$ the restriction of $P_{\nb}$ to $\ZF_{\ve}=\Si_{\nb}\setminus \Gc_\ve$. In this way, operator $\Tb_{P}$  splits into the sum
 \begin{equation}\label{Splitting}
 \Tb_{P}=\Tb_{\check{P}}+\Tb_{P'_{\nb,\ve}}+\Tb_{P_{\nb,\ve}}.
 \end{equation}
 In this splitting, the first operator is constructed by means of the measure supported on the union of $\nb-1$ Lipschitz surfaces, so the inductive assumption applies and the eigenvalue asymptotic formula of the type \eqref{SumMeasures} is valid. In the second operator, only one Lipschitz surface $\Si_\nb$ is involved, so by the base of induction, the asymptotic eigenvalue formula is holds as well. Note now that the measures in these two terms are supported in sets whose distance is at least $\ve$. Therefore \eqref{separated supp} applies, and
 \begin{gather}\label{twoOf Three}
  \lim_{\la\to0}  \la n_{\pm} (\la,\Tb_{\check{P}}+\Tb_{P'_{\nb,\ve}})= \lim_{\la\to0} \la n_{\pm} (\la,\Tb_{\check{P}})+\lim_{\la\to0} \la n_{\pm} (\la,\Tb_{P'_{\nb,\ve}})=\\\nonumber
\Zb(d,\dF)\left[\int_{\XF_{\nb-1}}P_{\pm}(dX) +\int_{\Zc_{\ve}}V_{\nb(X),\pm}d\m_\nb(X)\right]
 \end{gather}
 The third term in \eqref{Splitting} is the operator associated with measure $P_{\nb,\ve}$, i.e., supported in the part of $\Si_{\nb}$ lying in the $\ve$-neighborhood of $\XF_{\nb-1}$ but outside $\XF_{\nb-1}$. By  Theorem \ref{Th.As.RSh}, for the eigenvalues of this operator, the estimate holds,

 \begin{equation}\label{rem.est}
 n_{\pm}(\la,\Tb_{P_{\nb,\ve}} )\le C\la^{-1}\|V\|^{av,\Psi, }_{\YF_\ve}.
\end{equation}
Now, by choosing $\ve$  small enough, we can make the coefficient in \eqref{rem.est} arbitrarily small. Thus, again we can apply asymptotic perturbation Lemma 1.5 in \cite{BS}, which enables to pass to limit as $\ve\to 0$ on the left-hand side in \eqref{twoOf Three}, obtaining the left-hand side in \eqref{SumMeasures}.   The same passage to limit on the right-hand side in \eqref{twoOf Three} produces the required  quantity on the right-hand side in \eqref{SumMeasures}.
  \end{proof}
\subsection{Finite unions of Lipschitz surfaces of different dimensions.} Let, for each  $d=1,\,\dots,\Nb,$ a finite collection of compact Lipschitz surfaces be given,
$\Si_j^d,$ $1\le d\le\Nb,$ $j\le j_d.$ For $d=\Nb,$ a bounded open  set in  $\R^\Nb$ acts as $\Si_1^\Nb$. Let real densities $V_j^d(X)$ be given on surfaces $\Si_j^d$,
\begin{equation}\label{Cond.surface}
 V_j^d\in L^{\Psi,\m_j^d}(\Si_{j}^d),
 \end{equation}
where $\m_j^d=\m_{\Si_j^d}$ is the $d$-dimensional Hausdorff measure on $\Si_j^d.$ We consider measures $P_j^d= V_j^d\m_j^d.$ For $d=\Nb$ such measure is absolutely continuous with respect to the Lebesgue measure in $\R^\Nb;$ for $d<\Nb$ measures $P_j^d$ are singular.

We denote $P=\sum_{j,d}P^{d}_j$ and introduce the corresponding operator $\Tb_{P}.$ By considering quadratic forms, we immediately see that, under our conditions, this operator is bounded and equals the sum of operators $\Tb_{P_{j}^d};$ since each of the latter operators is compact, the same is correct for  $\Tb_{P}.$

In \cite{RSh2} we demonstrated some examples of measures with both absolutely continuous and singular components present. Generally, for $\AF=\AF_0$, the spectral problem for operator $\Tb$ with such measures is equivalent  to the eigenvalue problem for (pseudo) differential operator, containing the spectral parameter both in the equation and in transmission conditions on surfaces $\Si_j^d$ of  dimensions $d<\Nb$.

First of all, by the Ky Fan theorem, we obtain automatically the eigenvalue estimates for $\Tb_{P},$
\begin{equation}\label{Sum.diff.est}
    n_{\pm}(\la,\Tb_{P}) \le C \sum_{d,j}\|V_j^d\|^{\Psi,\m_j^d}\la^{-1}.
\end{equation}
The constant in \eqref{Sum.diff.est} depends on the quantity of surfaces present and is of no interest for us at the moment.

The results obtained in Sections \ref{subs.samedim.trace}, together with the linearity of the singular trace,  lead immediately to the integrability statement.
\begin{thm}\label{trace.sum.diff} Let $V_j^d$ satisfy condition \eqref{Cond.surface}. Then operator $\Tb_{P}$ is Connes measurable
and
\begin{equation}\label{Summa.trace}
    \ta(\Tb_{P})=\sum_{d,j}{\Zb(d,\Nb-d)}\int_{\Si_{j}^d}P(dX)=\int \NF(P)(dX).
\end{equation}
\end{thm}

The proof of the result on the eigenvalue asymptotics  takes a little bit more work. We show that contributions of components of measure $P$ supported on surfaces of different dimension add up in the asymptotic formula.
\begin{thm}\label{Th.SumAs}In conditions of Theorem \ref{trace.sum.diff},
we denote by  $\XF_{d}$ the set $\cup_j\Si_j^d,$ and introduce measure $P^d=\sum_{j}P^d_j$, as in Theorem \ref{Th.Add.Asymp}. Then
\begin{equation}\label{As.Sum.dif}
    n_{\pm}(\la,\Tb_P )\sim \la^{-1}\sum_d \Zb(d,\Nb-d)\int_{\XF_d}P^d_{\pm} =\la^{-1}\int\NF(P_{\pm})(dX).
\end{equation}
\end{thm}
\begin{proof}    We suppose, for simplicity, that all $\XF_d$ are nonempty. The reasoning then is similar to the one in Theorem \ref{Th.Add.Asymp}. We show that by cutting away arbitrarily small pieces $\YF_d$ of $\XF_d$, in the sense of $\Hc^d$ measure, we can make the remaining parts of $\ZF_{d}=\XF_d\setminus \YF_d$ separated. As soon as this is done, similarly to  \eqref{twoOf Three}, for  operator $\Tb_{P}$, the
leading contributions to the eigenvalue asymptotics corresponding to the restrictions of measures $P_d$ to $\ZF_d$ add up, while the contribution by these measures restricted to $\YF_d$ are small, and, again, the asymptotic perturbation lemma applies.

It remains to construct the sets $\F_d$. Consider $\Gc_1(\de_1),$ the  $\de_1$-neighborhood of $\XF_1$ in $\R^\Nb$. The Lebesgue measure of $\Gc_1(\de_1)$ tends to zero like $\de_1^{\Nb-1}$ as $\de_1\to 0$. Therefore for sufficiently small $\de_1,$ the portion of $\XF_2$ in $\Gc_1(\de_1)$ has $ \Hc^2$ measure smaller than a prescribed $\ve$. So we set $\ZF_1=\XF_1$ and $\ZF_2=\XF_2\setminus \Gc_1$. Then we take a $\de_2$-neighborhood $\Gc_2(\de_2) $ of $\ZF_2$ (the latter, recall, has Hausdorff dimension $2$). We take $\de_2$ such small that $\Gc_2(\de_2)\cap \XF_3$ have corresponding $\Hc^3$-measure less than $\ve$ and set $\ZF_3=\XF_2\setminus \Gc_2$. We continue this procedures in all dimensions removing a piece of small measure on each step so that the remaining sets $\ZF_d$ are separated. The smallness of the Hausdorff measures of the sets $\YF_d$ implies smallness of averaged Orlicz norms of densities $V_d$ over these sets. This leads to eigenvalue estimates by Theorem \ref{EstTheor}, an, finally, to the eigenvalues asymptotics by Lemma \ref{BSLemma}.
\end{proof}
\section{Connes measurability and rectifiable sets}\label{Rect.Sect}
In this section we extend the measurability and asymptotics results to measures supported on rectifiable sets. Such sets form an important topic in Geometric Measure Theory.
\subsection{Densities of measures} We recall here some key definitions and facts, \cite{Fal}, \cite{Mat1},  \cite{Fed},  \cite{Preiss}, \cite{De Lellis}  being our main reference sources.
\begin{def}\label{def.rect} A compact set $\XF\subset \R^\Nb$ is $d$-rectifiable if there exist a finite or countable collection of subsets $\Ac_j\subset \R^{d}$ and Lipschitz mappings $\pmb{\f}_{j}:\Ac_j\to \R^{\Nb}$ so that  $\Hc^{d}(\XF\setminus\cup_j\pmb{\f}_{j}(\Ac_j) )=0.$
\end{def}
In other words, $\XF$ should be, up to a set of zero Hausdorff measure, the union of not more than countably many Lipschitz surfaces.

An extensive literature deals with criteria for a set to be rectifiable. Sufficient conditions for rectifiability are usually expressed in terms of $s$-densities.
Let $\m$ be a finite Radon measure on $\R^{\Nb}$, $\XF=\supp\m$. For a point $X\in\XF,$ the upper and lower densities of order $s\in(0,\Nb]$ at $X$ are defined as
\begin{equation}\label{Dens.def}
    \Th^{*s}(\m,X)=\limsup_{r\to 0}r^{-s}\m(B(X,r));\, \Th^{s}_{*}(\m,X)=\liminf_{r\to 0}r^{-s}\m(B(X,r))
\end{equation}
(the infinite and zero values are allowed.) If these densities coincide, their common value, $ \Th^{s}(\m,X)$, is called the \emph{density} of order $s$ at $X$. In case of $\m$ being the $s$-dimensional Hausdorff measure, we replace $\m$ by $\XF$ in these notations.
 \begin{rem}\label{RemDens}Of course, if  measure $\m$ is $s$-Ahlfors regular then $0< \Cb \le\Th^{s}_{*}(\m,X)\le \Th^{s*}(\m,X)\le\Cb^{-1}$ for all $X\in\XF,$ where $\Cb$
 is the constant in \eqref{Aregular}. \end{rem}
We are interested in compact sets further on. The case when densities coincide is dealt with by \emph{ Marstrand's theorem}.
\begin{thm}[Theorem 14.10 in \cite{Mat1}]\label{Marst} Suppose that for a certain $s ,$ there exists a Radon measure $\m$ on $\R^{\Nb}$ such that for $X$, $\m$-almost everywhere, the upper and lower density at $X$ coincide and, moreover, their common value is finite and nonzero. Then $s$ is an integer, $s=d\in\N$.
\end{thm}
\subsection{Rectifiability conditions}
These conditions  can be found, e.g., in Sections 14-17 in \cite{Mat1} and Ch.3 in \cite{Fed}. We present here just a few, using density terms, see
\cite{Preiss}.
\begin{thm}[Density condition]\label{Rect.thm}The Borel set $\XF$  is rectifiable if and only if the density $\Th^d(\XF,X)$ exists, is positive and finite for $\Hc^d$-almost all $X\in\XF,$
\begin{equation}\label{Mat.Cond}
    0<\Th^d(\XF,X)<\infty,
\end{equation}
\end{thm}

The condition in Theorem \ref{Rect.thm} can be, at least formally, relaxed to the following, see \cite{Preiss} Corollary 5.5:
\begin{thm}\label{upper/lower} There exists a constant $\cb=\cb(d,\Nb)$ such that a Borel set $\XF\subset\R^{\Nb}$  is $d$-rectifiable if and only if the upper and lower densities satisfy
 \begin{equation}\label{PreissCond}
 0<\Th^{*d}(\XF,X)<\cb(d,\Nb)\Th^d_{*}(\XF,X)<\infty
 \end{equation}
 for $\Hc^d$-almost all $X\in\XF$.
\end{thm}

Rectifiability is an especially common property and is especially easy to check for sets having Hausdorff dimension $1$, see, e.g.,  \cite{Fal}, Theorem 3.11.
\begin{thm}\label{Dim1} Suppose that $\XF$ is a compact connected set  of Hausdorff dimension $1$ in $\R^\Nb$. Then $\XF$ is rectifiable.
\end{thm}
Of course, Theorem \ref{Dim1} extends automatically to the countable union of disjoint compact connected sets.
\subsection{Connes integration and eigenvalue asymptotics on rectifiable sets }\label{Rect.as}

Here we obtain main results of the paper.
\begin{thm}\label{Th.Connes.Rect} Let $\XF\subset\R^\Nb$ ba a rectifiable set of dimension $d>0$, so one of conditions \eqref{Mat.Cond}, \eqref{PreissCond} is satisfied; for $d=1$, $\XF$ is, instead, supposed to be a countable union of disjoint compact connected sets. Assume that the Hausdorff measure $\Hc^d$ on $\XF$  is Ahlfors regular of order $d$.
 If $V$ is a real-valued function on $\XF$ belonging to the Orlicz space $L^{\Psi,\m}(\XF)$ with respect to the Hausdorff measure $\m=\Hc^d,$ $P=V\m,$
then   operator $\Tb=\Tb_{P}$ is measurable and the Connes integration formula
\begin{equation}\label{Connes.Rectif}
    \ta(\Tb)=\Zb(d,\dF)\int_{\XF} V d\m=\int_{\XF}\NF(P)(dX)
\end{equation}
is valid.
\end{thm}
\begin{proof}
Let $\Si_j$ be a numeration of the Lipschitz surfaces entering in the definition of a rectifiable set. By the conditions of Theorem, this numeration can be chosen in such way that
\begin{equation*}
    \lim_{\nb\to \infty}\m(\XF\setminus\bigcup_{j<\nb}\Si_j)=0.
\end{equation*}
We define densities $V_j\in L^{\Psi,\m_j}_{\Si_j}$ in the following way.
For $j=1$, we set $V_1$ as the restriction of $V$ to $\Si_1$. Then, inductively, for $\nb>1$ we take $\XF_{\nb-1}=\cup_{j<\nb}\Si_j.$
Having $V_j,\, j<\nb$  defined, we set $V_{\nb}$ as the restriction of $V$ to the set $\Si_{\nb}\setminus\XF_{\nb-1}.$ Constructed in this way, for any point $X\in\XF,$
no more than one of functions $V_j$ is nonzero. Moreover,
 $\sum_j{V_j}=V,$ pointwise and in $L^{\Psi,\m}$, with $\sum_jV_{\pm}=V_{\pm}.$
 By Theorem \ref{FinitProp}, for operator $\Tb_{V,\m,\XF_\nb}$ involving finitely many Lipschitz surfaces, the Connes integration formula is valid
\begin{equation}\label{int.part}
    \ta(\Tb_{V,\m,\XF_\nb})=\Zb(d,\dF) \int_{\XF_\nb}V d\m.
\end{equation}
On the other hand, by Theorem \ref{EstTheor},
\begin{equation*}
    \limsup_{\la \to 0}\la n_{\pm}(\la, \Tb_{P,\XF\setminus\XF_{\nb}})\le C\|V\|^{(av,\Psi,\m)}_{\XF\setminus\XF_{\nb}}.
\end{equation*}
The quantity on the right tends to zero as $\nb\to\infty.$ Therefore, by continuity, the singular trace of operator $\Tb_{P,\XF\setminus\XF_{\nb}}$ tends to zero as $\nb\to\infty$ and it is possible to pass  to limit in \eqref{int.part} which gives \eqref{Connes.Rectif}.
\end{proof}

Similar to the previous section, the statement about the eigenvalue asymptotics, obviously, a stronger one than the integration formula, is valid.

\begin{thm}\label{Th.Asymp.rectif} Let $\XF$ be a compact rectifiable set of dimension $d$ in $\R^{\Nb}$, Ahlfors $d$-regular, and $V$ be a real Borel  function on $\XF$ such that $\|V\|^{\Psi,\m}_{\XF}<\infty,$ $\m=\Hc^d$ is the $d$-dimensional Hausdorff measure. Then for the eigenvalues of $\Tb_{P,\XF},$ $P=V\m,$ the following asymptotic formula is valid.
\begin{equation*}
    \la n_{\pm}(\la, \Tb_{V,\m,\XF})\sim \Zb(d,\dF)\int_{\XF}V_{\pm}d\m=\int_{\XF}\NF(P_{\pm})(dX).
\end{equation*}
 \end{thm}
 \begin{proof} The proof is a copy of the reasoning above for Theorem \ref{Th.Connes.Rect}. Just instead of the asymptotic relation \eqref{int.part}, we use on each stage the asymptotic formula \eqref{SumMeasures} for the eigenvalues and pass to the limit as $\nb\to\infty$ by means of Lemma \ref{BSLemma}.
 \end{proof}
 \subsection{Unions of rectifiable sets of different dimension}\label{Diff.dim.}
 At last, we consider  the most general case. Let  measure $\m$ on $\R^{\Nb}$ be the sum of measures $\m_d$, $1\le d\le \Nb$, each of them being the Hausdorff measure $\Hc^d$ of the corresponding dimension, restricted to a compact rectifiable set $\XF_d$ of dimension $d$, $\XF=\cup \XF_d$. Having densities $V_d\in L^{\Psi,\m_d}(\XF_d)$, we consider signed measures $P_d=V_d\m_d,$ their sum $P=\sum P_d$ and the corresponding operator $\Tb_P=\sum_d \Tb_{P_d}$.
 \begin{thm}\label{FinalThm.general} Let $\XF\subset\R^{\Nb}$ and operator $\Tb_P$ be as above. Then
 \begin{enumerate}\item{i} Operator $\Tb_P$ is Dixmier-Connes measurable and for any normalized singular trace $\ta,$
 \begin{equation*}
    \ta(\Tb_P)=\sum_d\Zb(d,\Nb-d)\int_{\XF_d} P_d(dX)=\sum\int_{\XF_d}\tilde{P}_d(dX)=\int_{\XF} \NF(P)(dX).
 \end{equation*}
\item{ii} For  operator $\Tb_P$ the eigenvalue asymptotic formulas hold
 \begin{equation*}
    n_{\pm}(\la, \Tb_P)\sim \la^{-1}\sum_d \Zb(d,\Nb-d)\int_{\XF_d} P_{\pm,d}(dX) \sim \la^{-1}\int_{\XF}\NF(P_{\pm})(dX).
 \end{equation*}
 \end{enumerate}
 \end{thm}
\begin{proof} The statement (i) follows from (ii). Alternatively, it is obtained from  \eqref{int.part} due to the linearity of the singular trace by summing over $d$. Statement (ii) is established similar to Theorem \ref{Th.Asymp.rectif}. For a given small $\ve$, by definition, we can find finite collections $\Si_j^d$ of Lipschitz surfaces of dimension $d$ such that the Hausdorff measure of corresponding dimension of the set  $\XF_d\setminus \cup_j \Si_j^d$ is less than $\ve$, together with the averaged Orlicz norm of $V_d$ restricted the latter set. To the operator corresponding to the union of finitely many surfaces $\Si_j^d ,$ Theorem \ref{Th.SumAs} applies, giving the eigenvalue asymptotic formula. The remainder by the smallness of the Orlicz norms, satisfies an eigenvalue estimate with small constant. Finally, Lemma \ref{BSLemma} produces the required result in the usual way.
\end{proof}

\section{Spectral problems on Riemannian manifolds}\label{Sect.surface} Consider a closed smooth  Riemannian manifold $\Mc$ of dimension $\Nb$. Denote by $\D$ the Laplace-Beltrami operator on $\Mc$. A compact subset $\Si\subset\Mc$ is called Lipschitz surface of dimension $d$ if its image under co-ordinates mappings  are $d$-dimensional Lipschitz surfaces in domains in the Euclidean space. The results presented in previous sections carry over to this setting by a simple localizations, using Proposition \ref{Prop.Local}. We present here some calculations needed for this case.

Let $\gb=\{g_{\a\be}(X)\}$ be the metric tensor of $\Mc$  in some local co-ordinate system, $\gF$ will denote the inverse matrix, $\gF=\gb^{-1}=\{g^{\a\be}\}$. The Laplace-Beltrami operator $\D$ is the second order elliptic operator with principal symbol $\db(X,\X)= -\sum_{\a,\be} g^{\a\be}(X)\X_\a\X_\be$, thus the principal symbol of operator
$\AF=(1-\D)^{-l/2}$ equals $\ab_{-l}(X,\X)=(\sum_{\a,\be} g^{\a\be}(X)\X_\a\X_\be)^{-\Nb/4}.$

Let further the surface $\Si$, in the some local co-ordinates $X=(\xb,\yb)\in \R^{d}\times\R^{\dF}, $ $\dF=\Nb-d$, be defined by $\yb=\pmb{\f}(\xb),$ where $\pmb{\f}$ is a Lipschitz $\dF$-component vector-function. The above embedding $F:\xb \mapsto (\xb,\pmb{\f}(\xb))$ of $\Si$ into $\Mc$ generates a (nonsmooth) Riemannian metric $\hb$ on $\Si,$ namely, $\hb(\z,\z'):=\gb(DF(\z),DF(\z'))$ for tangent vectors $\z,\z'$ to $\Si$. Here $DF=(\pmb{1},\nabla\pmb\f)$ is the differential of the embedding $F,$ defined at the points of $\Si$ where  $DF$ exists, i.e., almost everywhere with respect to the Lebesgue measure on $\Si$ in local co-ordinates. Further calculations are being made just in such points.  Having the Riemannian metric on $\Si,$ the Riemannian measure $\m=\m_\Si$ is defined, in co-ordinates $\xb,$ as $\m_\Si=H(\xb)^{\frac12}d\xb,$ where $H(\xb)=\det(\hb),$ or, more explicitly,  $H(\xb)=\det\left(\pmb{1}+\gb(\nabla \pmb{\f},\nabla\pmb{\f})\right).$

As it was done in the Euclidean case, we consider a  density $V(X), \, X\in \Si$ and operator $\Tb_{P,\AF}=\AF^* P\AF$ in $L_2(\Mc),$ $P=V\m_{\Si},$ with respect to the Riemannian measure on $\Mc.$ Using eigenvalue estimates obtained in Section \ref{Sect.Est}, we, as before,  reduce the problem of finding asymptotics of eigenvalues of $\Tb_{P,\AF}$ to the case of the density $V$ being a nonnegative smooth function on $\Mc,$ $V=U^2$. Thus the operator $\Tb_{P,\AF}$ factorizes, similar to \eqref{factorizationSi},
as
\begin{equation}\label{Factor.last}
\Tb_{P,\AF}=(\G_{\Si}U\AF)^*(\G_{\Si}U\AF)=\KF^*\KF,
\end{equation}
where $\G_{\Si}$ is the operator of restriction from the Sobolev space $H^l(\Mc)$ to $L^2(\Si)$. As before, we note that  the nonzero eigenvalues of $\KF^*\KF$
 coincide with the ones of $\KF\KF^*,$ the latter being an integral operator on $\Si,$ the restriction to $\Si$ of the pseudodifferential operator $U\AF\AF^*U.$ The principal symbol of this operator equals $V(X)\ab_{-l}(X,\X)^2= V(X)(\sum_{\a,\be} g^{\a\be}(X)\X_\a\X_\be)^{-\Nb/2}$. By the results of \cite{RT2}, it suffices to consider the case of a smooth surface $\Si.$ Now, the restriction of  pseudodifferential operator $U\AF\AF^*U$ to  surface $\Si$ is performed according to the rules explained in Section \ref{Subsect.Lip}, following \cite{RT2}. Namely we calculate the symbol $r_{-d}(X,\x),$ $X\in \Si, \x\in T^{*}_X\Si$ by the rule in \eqref{Th.As.RSh}, which produces a symbol on $\Si,$ $r_{-d}(X,\x)= R(X,\x)^{-d/2},$ where $R(X,\x)$ is, for each fixed $X$, a quadratic form in $\x$ variables. Thus, $\KF^*\KF$ is an integral operator on $\Si$ with the leading singularity of the  kernel being equal to the Fourier transform of symbol $r_{-d}$ in $\x$ variable. Such Fourier transform   $\Rc(X,Y,X-Y)$ has logarithmic singularity in $X-Y$,
 \begin{equation}\label{Symb.manif}
    \Rc(X,Y)=R(X)\log(Q_X(X-Y)) +o(|X-Y|),
 \end{equation}
with certain quadratic form $Q_X$. Finally we arrive at the same asymptotic formula for eigenvalues as in the 'flat' case. This gives us the required versions of the results of Section \ref{Sect.As} for operators on Riemanninan surfaces.
 \begin{thm}\label{Th.As.Riem}  Let $\Si$ be a $d$-dimensional compact Lipschitz surface in an $\Nb$- dimensional Riemannian manifold $\Mc$, $\m_{\Si}$ be a measure on $\Si$ generated by the embedding of $\Si$ into $\Mc$. For a real function $V\in L^{\Psi,\m_\Si},\, P=V\m_\Si$ consider the operator $\Tb_{P,\AF}=\AF P\AF,$ where $\AF=(1-\D)^{-\Nb/4}.$
 Then operator $\Tb_{P,\AF}$ is Connes measurable,  and its eigenvalue asymptotics is given by formulas \eqref{AsLambda}.
 \end{thm}
 Results on the spectral properties of measures on rectifiable sets, see Sections \ref{Rect.as},  \ref{Diff.dim.} are carried over to the setting of Riemannian surfaces in the same way.

 \section{Lower estimates}\label{lower}
 In the results presented above, certain asymmetry is present. While the upper eigenvalue estimates for operators of the form $\Tb_{P,\AF}$ are established for measures supported on Ahlfors regular sets of any dimension $0<d\le\Nb,$ the eigenvalue asymptotics is proved only for rectifiable sets, thus, only for sets that have integer Hausdorff dimension, and even for not all of them. Therefore, the natural question arises  about order sharpness of our upper estimates. This section is devoted to establishing this sharpness. It turns out that lower estimates for eigenvalues can be justified in even more general setting than the upper ones.
 \begin{thm} Let $\AF$ be an order $-l=-\Nb/2$ pseudodifferential operator in $\Om\subset\R^\Nb$, elliptic in a domain $\Om'\subset\Om,$ and $\m$ be a finite Borel measure with compact support inside $\Om'.$ Suppose that $\m$ does not contain atoms and the density $V\ge0$ satisfies $\int V(X) \m(dX)<\infty$. Then for $P=V\m,$
 \begin{equation}\label{lowerEStim}
  \liminf \la  n_{+}(\la, \Tb_{P,\AF}) \ge [C(\AF)P(\Om)]=[C(\AF)\int Vd\m],
 \end{equation}
 where brackets on the right-hand side  denote the integer part of the number inside. In  inequality \eqref{lowerEStim}, the expression on the left-hand side is set to be equal to $+\infty$ if operator $\Tb_{P,\AF}$ is unbounded.
 \end{thm}
 \begin{proof} By ellipticity, it is sufficient to consider the case $\AF=(1-\Delta)^{-\Nb/4}$, with cut-offs to $\Om$.

  With measure $P$ we associate the quadratic form $\qb_{tP}[v]=t\int |v(X)|^2 P(dX),$ $t>0,$ and consider the Schr\"odinger-type quadratic form $\hb_t[v]= \ab[v] -\qb_{tP}[v], $ $\ab[v]= \|v\|^2_{H^l(\Om')}.$ By the Birman-Schwinger principle,
 if this form, for certain $t>0,$ is lower semibounded, and thus defines a self-adjoint operator $\HF_t$, the number of negative eigenvalues of this operator is no greater than the number of eigenvalues  of $\Tb_{P,\AF}$ in $(t^{-1},\infty)$,
 $N_-(\HF_t)\le n_+(t^{-1},\Tb_{P,\AF}).$ We need to consider only such (not that large) values of $t$ since if the quadratic form $\hb_t$ is \emph{not} lower semibounded, the quantity $n_+(t^{-1},\Tb_{P,\AF})$ is infinite and \eqref{lowerEStim} is satisfied automatically.

  Let first $\Nb$ be an even number, so $l$ is integer. Then the form $\ab$ is equivalent to $\int_{\Om'}|\nabla^l v|^2dX+\|v\|^2$, and this form is local. So, we are in the conditions of the paper \cite{GNY}, see Theorem 4.1 and Example 4.13 there, which gives estimates from below for the number of negative eigenvalues of the form $\hb_t$, exactly as in \eqref{lowerEStim}.

 For an odd dimension $\Nb,$ i.e., for a non-integer $l,$ a direct application of Theorem 4.1 in \cite{GNY} is impossible since this theorem requires the form $\ab[v]$ to be local. Therefore, we use the trick of dimension lift, compare with \cite{ShStokes},  see proof of Theorem 1.2 there. Consider the space $\R^\Nb$ being embedded into $\R^{\Nb+1}$ as an $\Nb$-dimensional linear subspace.
 For $l=\Nb/2$, there exists a bounded restriction operator $\Tr:H^{l+\frac12}(\R^{\Nb})\to H^{l}(\R^\Nb),$ so that $\|\Tr v\|_{H^{l}(\R^\Nb)}\le c_0\|v\|_{H^{l+\frac12}(\R^{\Nb+1})}.$ With measure $P$ on $\R^{\Nb},$ we associate measure   ${P}^\divideontimes=P\otimes\delta_{X_{\Nb+1}}$ on $\R^{\Nb+1}$. The quadratic form ${\ab}^\divideontimes[v]=\|v\|^2_{H^{l+\frac12}(\R^{\Nb+1})}$ is now local, and thus we can apply Theorem 4.1 in \cite{GNY} to the form ${\hb}^\divideontimes_t={\ab}^\divideontimes[v]-t\int|v|^2{P}^\divideontimes(dX), $  obtaining
 \begin{equation}\label{GNY1}
  \nb_t:=  N_-({\hb}^\divideontimes_t)\ge [c t {P}^\divideontimes(\R^{\Nb+1})]=[c t P(\R^{\Nb})].
 \end{equation}
 By the min-max principle, this means that there exists a subspace $\Lc\subset H^{l+\frac12}(\R^{\Nb+1}),$ $\dim\Lc=n+t$ such that
 \begin{equation*}
{\hb}^\divideontimes_t[v]<0, \, v\in \Lc\setminus\{0\},
 \end{equation*}
 and $\dim \Lc=\nb_t,$ or
 \begin{equation}\label{GNY3}
  {\ab}^\divideontimes[v]<t\int_{\R^{\Nb}} |v|^2P(dX).
 \end{equation}
Due to denseness of continuous function in $H^{l+\frac12},$ we can suppose that $\Lc$ consists of continuous functions.
 Consider the subspace $\Tr(\Lc)\subset H^{l}(\R^{\Nb})\cup C(\R^{\Nb})).$ It has the same dimension as $\Lc$. In fact, if some nonzero function $v\in\Lc$ is annulled by $\Tr$, $\Tr v=0$, this would mean that $v$ is zero on $\R^{\Nb}$ and therefore $\int_{\R^{\Nb}} |v|^2P(dX)=0,$ which contradicts \eqref{GNY3}. Thus, the mapping $\Tr$ is injective on $\Lc$. We denote by $\Ec$ its right inverse mapping $\Ec: \Tr(\Lc)\to \Lc$. Therefore
 \begin{gather}\label{GNY4}
    {\ab}[v]-c_0t\int_{\R^{\Nb}} |v|^2(\R^{\Nb})P(dx)\le c_0\left[ {\ab}^\divideontimes[\Ec v]-t\int_{\R^{\Nb}} |\Ec v|^2(\R^{\Nb})P(dx)\right]<0,\\ \nonumber v\in \Tr(\Lc), \, v\ne -0.
 \end{gather}
 By the variational principle, \eqref{GNY4} means that the number of negative eigenvalues of  operator $\Hb_{c_0 t}$ is no less than $\nb_t$ in \eqref{GNY1}, which  proves our Theorem for this case as well.
 \end{proof}

\begin{acknowledgments}
The work is supported by Ministry of Science and Higher Education of the Russian Federation, agreement  075--15--2019--1619.
\end{acknowledgments}

\end{document}